\documentclass [11pt]{article}

\usepackage{amssymb}
\usepackage{amsmath}
\usepackage{amsfonts,mathrsfs}
\usepackage{amsthm}
\usepackage{srcltx}
\usepackage{epsfig}
\usepackage[dvipsnames,usenames]{color}
\usepackage{paralist}

\newcounter{contador}
\newtheorem{propo}[contador]{Proposition}
\newtheorem{teo}[contador]{Theorem}
\newtheorem{lem}[contador]{Lemma}

\newcounter{teoremaganso}

\newcommand{\R}{{\mathbb R}}

\newcommand{\N}{{\mathbb N}}
\newcommand{\Q}{{\mathbb Q}}

\title{Periodic points of a Landen transformation}

\author{Armengol Gasull$^{(1)}$, Mireia Llorens$^{(1)}$ and V\'{\i}ctor Ma\~{n}osa $^{(2)}$
  \\*[.1truecm]
{\small \textsl{$^{(1)}$ Departament de Matem\`{a}tiques, Facultat
de Ci\`{e}ncies,}}
\\*[-.25truecm] {\small \textsl{Universitat Aut\`{o}noma de Barcelona,}}
\\*[-.25truecm] {\small \textsl{08193 Bellaterra, Barcelona, Spain}}
\\*[-.25truecm] {\small \textsl{mllorens@mat.uab.cat,
gasull@mat.uab.cat}}
\\*[-.1truecm] {\small \textsl{$^{(2)}$ Departament de Matem\`{a}tiques}}
\\*[-.25truecm] {\small \textsl{Universitat Polit\`{e}cnica de Catalunya}}
\\*[-.25truecm] {\small \textsl{Colom 1, 08222 Terrassa, Spain}}
\\*[-.25truecm] {\small \textsl{victor.manosa@upc.edu}}}
\begin{document}

\maketitle
\begin{abstract} We prove the existence of 3-periodic orbits  in a dynamical
  system associated to a Landen transformation previously studied by Boros,
  Chamberland and Moll, disproving a conjecture on the dynamics of
  this planar map introduced by the latter author.
To this end we present a systematic methodology to determine and
locate analytically
 isolated periodic points of algebraic maps. This approach can be
  useful to study other discrete dynamical systems with algebraic nature.  Complementary results
  on the dynamics of the map associated with the Landen transformation are also presented.
\end{abstract}

\noindent {\sl  Mathematics Subject Classification 2010:} 37C25, 33E05.

\noindent {\sl Keywords: Landen transformation, Periodic points,
Poincar\'{e}-Miranda theorem.}

\section{Introduction}

Given a definite integral depending on several parameters, a
\emph{Landen transformation} is a map on these parameters that
leaves invariant the integral. In   \cite{BM0,BM}, G.~Boros and
 V.~Moll introduced the dynamical system given by
$$
  \left\{ \begin{array}{l}
  a_{n+1}=\dfrac{5a_n+5b_n+a_nb_n+9}{(a_n+b_n+2)^{4/3}},\quad
  b_{n+1}=\dfrac{a_n+b_n+6}{(a_n+b_n+2)^{2/3}},\\
    \ \\
  c_{n+1}=\dfrac{d_n+e_n+c_n}{(a_n+b_n+2)^{2/3}},\quad
d_{n+1}=\dfrac{(b_n+3)c_n+(a_n+3)e_n+2d_n}{a_n+b_n+2},\quad
  e_{n+1}=\dfrac{c_n+e_n}{(a_n+b_n+2)^{1/3}},\\
  \end{array} \right.
$$
as a Landen transformation associated to the integral
 \begin{equation}\label{eq6}
 I(a,b,c,d,e)=\int_{0}^{\infty} {\frac{cx^4+dx^2+e}
 {x^6+ax^4+bx^2+1}\mathrm{d}x},
 \end{equation}
that is,
 $I(a_{n+1},b_{n+1},c_{n+1},d_{n+1},e_{n+1})=I(a_{n},b_{n},c_{n},d_{n},e_{n})$.
 This dynamical system contains a 2-dimensional uncoupled subsystem.  M.~Chamberland and
 V.~Moll in \cite{chamb}, related the convergence of the integral~\eqref{eq6} with the dynamics
 given  by the iteration of the planar, non invertible map
 associated to it:
$$
  G(a,b):=\left(\frac{5a+5b+ab+9}{(a+b+2)^{4/3}},
  \frac{a+b+6}{(a+b+2)^{2/3}}\right).
$$
In particular they proved that the map $G$ has only three fixed points,
characterizing  their nature, and they also proved that  the region of the $(a,b)$-plane where the
  integral \eqref{eq6} converges is the  basin of attraction of the fixed  point $(3,3)$.

In Section \ref{S-Landen} we will give a brief description of the known results about
 the dynamics of the map $G$. In Section
  \ref{S-fixos-i-periodics} we prove our main result,

\begin{teo}\label{T-main} Consider the map $G$.
  \begin{enumerate}
    \item[(a)] It has exactly three fixed points. A super-attracting point in
                $(3,3)$, an oscillatory saddle in the boundary of
                the basin of attraction of $(3,3)$ and an unstable focus.
    \item[(b)] It has  not periodic points with minimal period  $2$.
    \item[(c)] It has exactly twelve periodic points of minimal period\, $3$, that correspond with four $3$-periodic orbits.
  \end{enumerate}
 \end{teo}

For completeness we include in its statement the results about fixed
points already proved in~\cite{chamb}. In fact, in that paper   it
is also proved that there are no periodic points with minimal
period~$2$ above the line $a+b+2=0$. Our statement $(b)$ extends
their result to the whole plane.

 As we will see, item $(c)$
  disproves  a conjecture about the dynamics of this map, see \cite[Conj. 15.6.3]{llibre-Moll} or Section
  \ref{S-Landen}. We will also determine analytically the location of the $3$-periodic orbits.

Although it is easy to find $3$-periodic points  numerically, when
trying to prove their existence there appear important computational
obstacles. Thus, to prove the existence of $3$-periodic points of
$G,$ as well as the non-existence of $2$-periodic points, we have
developed a procedure to determine analytically the number of
isolated periodic points of discrete dynamical systems of algebraic
nature and locate them with a prescribed precision. This method
consists in the following four steps:

\begin{itemize}

\item  Convert the problem into an algebraic one, characterizing
the periodic points as the solutions of a system of polynomial
equations.

\item  Include these solutions into the ones of an uncoupled system
of equations given by one-variable polynomials.

\item Combine an algorithm based on the \emph{Sturm's method} for
isolating the
 real roots of a one-variable polynomial with a \emph{discard procedure} for
 systems of polynomial equations in order to efficiently remove those solutions
 of the later system that do not correspond with the periodic
 points.

 \item  The
 application of the  \emph{Poincar\'e-Miranda theorem} to prove that the non discarded
 solutions are actual solutions of the first system of polynomial equations
 and, in consequence, give rise to  periodic points.

\end{itemize}

  This procedure is explained in detail in
 next section. Recall that the Poincar\'{e}-Miranda
 theorem is essentially the extension of Bolzano theorem to higher
 dimensions. It was stated by H.~Poincar\'e in 1883 and 1884, and proved by
himself in 1886, \cite{Poinc1,Poinc3}. In 1940, C.~Miranda
re-obtained the result as an equivalent formulation of Brouwer fixed
point theorem, \cite{Miranda}. Recent proofs are presented in
\cite{K,V}. We also recall this theorem in Section
\ref{Ss-metode-general}.

As a complement, in Section \ref{S-Num-An} we characterize the
stable set associated to the fixed point of $G$ of saddle type, and
we provide an analytic-numeric study that gives evidences of the
existence of homoclinic trajectories associated to it, as well as of
the existence of some points in the intersection of the unstable set
of this fixed point and the non-definition set of the map, which
recall that it is formed by all the preimages of the straight line
$a+b+2=0.$

\section{Determination of periodic points of discrete dynamical systems}\label{Ss-metode-general}

We consider a discrete dynamical system defined by a map
$F:\mathcal{U}\subseteq\R^k\rightarrow\mathcal{U}$ where
$\mathcal{U}$ is an open set. Fix $p\in\N$ and assume that it has
finitely many $p$-periodic points. These points are characterized by
the real solutions of the system of $k$ equations given by
$F^p=\mathrm{Id}$. Let us suppose that the solutions of the above
system  are in correspondence with the ones of a new system of
$n\geq k$ non-trivial \emph{polynomial} equations given by
 \begin{equation}\label{desc1}
 \left\{ \begin{array}{c}
 f_1(\mathbf{x})=0,\, f_2(\mathbf{x})=0,\,
 \cdots,\,
 f_n(\mathbf{x})=0,
       \end{array} \right.
 \end{equation}
  where $\mathbf{x}=(x_1,\ldots,x_n)$ are not necessarily the $k$-independent variables of $F.$
   Suppose also
  that using some algebraic transformations, like for instance successive resultants between the given equations, we reach an
   uncoupled polynomial system whose set of solutions  \emph{contains} all the solutions
of system~\eqref{desc1}:
 \begin{equation}\label{desc2}
 \left\{ \begin{array}{c}
 q_1(x_1)=0,\,
 q_2(x_2)=0,\,
 \cdots,\,
 q_n(x_n)=0.
       \end{array} \right.
 \end{equation}

To clarify with an example the above situation we sketch  here the
systems involved in the computation of the
 $3$-periodic points of the map $G.$ The $k=2$ equations corresponding to $G^3(a,b)=(a,b)$
can be  transformed into a new  system of $n=3$ polynomial equations
(see system \eqref{E-d10d11d12}) in the new variables $m,n,r$ given
 by \eqref{E-abcdef}. This new system plays the role of
system \eqref{desc1}, and its solutions are in correspondence with
the periodic points, by forthcoming Lemma~\ref{O-1}.
Lemma~\ref{L-lem2} will show that the solutions of system
\eqref{E-d10d11d12} are included in the set of solutions of the
uncoupled system $\{d_{17}(m)=d_{17}(n)=d_{17}(r)=0\}$, where
$d_{17}$ is a polynomial of degree 371 introduced in~\eqref{eq:d17}.
This system plays the role of system~\eqref{desc2}.

In this setting, the proposed methodology applies in the cases where
we do not know how to obtain explicitly the solutions of systems
\eqref{desc1} or \eqref{desc2} and follows the next steps:

\smallskip

\noindent \textbf{Step 1:} By using an algorithm based on the
Sturm's method   (\cite[Chap. 5.6]{StB}) and for each polynomial
$q_j$, it is possible to isolate and count all its  real roots by
finding intervals with preset maximum length and rational ends, each
one of them containing only one isolated root.  For each
$j=1,2,\ldots,n,$ let $k_j$ be the number of real roots of $q_j,$
without counting their multiplicities, and denote by
$I_{j,m}:=[u_{j,m},v_{j,m}],$ $m=1,2,\ldots,k_j$ the found
intervals, such that each one of them contains exactly one of these
roots. Proceeding in this way we obtain that the set of solutions of
system \eqref{desc1} \emph{is contained} in the set formed by
$\prod _{j=1}^n k_j$  boxes ($n$-dimensional orthohedrons), of the form
$$\mathcal{I}_{m_1,\ldots,m_n}:=I_{1,m_1}\times I_{2,m_2}\times \cdots \times I_{n,m_n},$$ where each $m_j\in\{1,\ldots,k_j\}$, for $j=1,\ldots,n$.

\smallskip

\noindent \textbf{Step 2:} In order to detect those boxes that do
not contain any solution of system \eqref{desc1} we apply a
\emph{discard procedure} to each box $\mathcal{I}_{m_1,\ldots,m_n}.$ This procedure is
inspired in a technique used in~\cite{JD2}. To prove that a certain
polynomial
 $P(\mathbf{x})$ has no zeros in a given box $\mathcal{I}_{m_1,\ldots,m_n},$ that for the sake of simplicity
 we denote as $\mathcal{I},$  we proceed as follows:

  \begin{itemize}

\item We numerically evaluate $P$ at the center of $\mathcal{I}$. If,
compared with the working precision, this value is far from zero, we
suspect that $P$ restricted to $\mathcal{I}$ has a given sign. According
whether this value is positive or negative we continue with one of
next two steps.

  \item For trying to prove that $P(\mathbf{x})>0$ for all
        $\mathbf{x}\in   \mathcal{I}$, we search a $L$ such that
        $0<L<P(\mathbf{x})$ on $\mathcal{I}.$ Write
        $P(\mathbf{x})=\sum_\ell M_\ell(\mathbf{x})$  where $M_\ell(\mathbf{x})=a_\ell x_1^{\ell_1} x_2^{\ell_2}\cdots x_n^{\ell_n} $,
         we find $\underline{M}_{\,\ell}\in \mathbb{R}$ such that
        $\underline{M}_{\,\ell}<M_\ell(\mathbf{x})$ for all  $\mathbf{x}\in  \mathcal{I}$ (this can be done using
         the formulas in forthcoming Lemma \ref{L-minmaj-nou-n}). If the following condition is satisfied:
        $0<L:=\sum_\ell \underline{M}_{\,\ell}<\sum_\ell
        M_\ell(\mathbf{x})=P(\mathbf{x}),$
        then we can discard the box  $\mathcal{I}$.
  \item For trying to  prove that $P(\mathbf{x})<0$ for all
        $\mathbf{x}\in   \mathcal{I}$,  we look for  $U\in\R$  such that
        $P(\mathbf{x})<U<0$ on~$\mathcal{I}.$ To do this, similarly than in the previous situation,
         we find  $\overline{M}_{\,\ell}\in \mathbb{R}$ such that
        $M_\ell(\mathbf{x})<\overline{M}_{\,\ell}$ for all $\mathbf{x}\in   \mathcal{I}$. If
        it holds that
        $P(\mathbf{x})=\sum_\ell M_\ell(\mathbf{x})<\sum_\ell
        \overline{M}_{\,\ell}=:U<0,$
        then we can discard the box $\mathcal{I}$.
 \end{itemize}

To compute the bounds  $\underline{M}_{\,\ell}$ and
 $\overline{M}_{\,\ell}$, we use the following straightforward result, which can be easily implemented in any computer algebra software.

\begin{lem}\label{L-minmaj-nou-n} Consider
$P(\mathbf{x})=\sum_\ell M_\ell(\mathbf{x})$  where
$M_\ell(\mathbf{x})=a_\ell x_1^{\ell_1} x_2^{\ell_2}\cdots
x_n^{\ell_n} $, and a box $ \mathcal{I}=[u_1,v_1]\times
[u_2,v_2]\times\cdots\times [u_n,v_n].$ Set
$O^{+}=\{(x_1,\ldots,x_n),$  such that $x_i>0$ for all
 $i=1,\ldots,n\}$. Then,

\begin{enumerate}[(i)]
\item  If $ \mathcal{I}\subset O^{+}\subset\R^n,$ then for all $\mathbf{x}\in
\mathcal{I},$ $\sum _\ell   \underline{M}_{\,\ell}\le P(\mathbf{x}) \le \sum
_\ell
  \overline{M}_{\ell},$ where
\begin{enumerate}
\item[(a)] $\underline{M}_{\,\ell}=a_\ell\,u_1^{\ell_1} u_2^{\ell_2}\cdots u_n^{\ell_n}$
 and $\overline{M}_\ell=a_\ell\,v_1^{\ell_1} v_2^{\ell_2}\cdots v_n^{\ell_n}$ if  $a_\ell>0$.
\item[(b)] $\underline{M}_{\,\ell}=a_\ell\,v_1^{\ell_1} v_2^{\ell_2}\cdots v_n^{\ell_n}$
 and $\overline{M}_{\ell}=a_\ell\,u_1^{\ell_1} u_2^{\ell_2}\cdots u_n^{\ell_n}$ if  $a_\ell<0$.
\end{enumerate}

\item If $\mathcal{I}\not\subset O^{+}$ we can always  take a number $\xi>0$,
$\xi\in\mathbb{Q}$ such that the new box
$\widetilde{\mathcal{I}}=[u_1+\xi,v_1+\xi]\times
[u_2+\xi,v_2+\xi]\times\cdots\times [u_n+\xi,v_n+\xi]\subset O^{+},
$ and then find bounds for $P$ on $\mathcal{I},$ using the bounds given in
item $(i)$ for $\widetilde P_\xi(x_1,x_2,\ldots,x_n):=
P(x_1-\xi,x_2-\xi,\ldots,x_n-\xi)$ on  $\widetilde{\mathcal{I}}.$
\end{enumerate}
\end{lem}

We try to apply the discard procedure until the number of remaining
boxes coincides with our hopes. These hopes usually came from a
previous numerical study of the problem. We start trying to prove
that  the first function $f_1$ does not vanish in the given box. It
may happen that it is easier to try to prove the same with another
$f_j.$ Notice also that sometimes to discard a box we must go to the
Step 1 and start with smaller boxes.

\smallskip

\noindent \textbf{Step 3:} Once it is achieved an optimized list of
        non-discarded boxes, we  identify those boxes that correspond to either fixed points or periodic points
        with a period being a divisor of $p$, which we assume that we already know, and we also   discard them.

\smallskip

\noindent \textbf{Step  4:} From the non-discarded boxes list obtained in the previous step,
 we  try to show that  each box actually contains a solution by applying the  Poincar\'e-Miranda
theorem. For completeness, we recall it. As usual, $\overline B$ and
$\partial B$ denote, respectively, the closure and the boundary of a
set $B\subset\R^n.$

\begin{teo}[Poincar\'e-Miranda]\label{T-PM-n}
       Set $\mathcal{I}=\{\mathbf{x}=(x_1,\ldots,x_n)\in\R^n\,:\,L_i<x_i<U_i, 1\leq i\leq n\}$. Suppose that
         $f=(f_1,f_2,\ldots,f_n):\overline{\mathcal{I}}\rightarrow R^n$ is
         continuous,  $f(\mathbf{x})\neq\mathbf{0}$
           for all $\mathbf{x}\in\partial \mathcal{I}$, and for  $1\leq i\leq
       n,$
       $$f_i(x_1,\ldots,x_{i-1},L_i,x_{i+1},\ldots,x_n)\cdot f_i(x_1,\ldots,x_{i-1},U_i,x_{i+1},\ldots,x_n)\leq 0,$$ 
       Then, there exists $\mathbf{s}\in\mathcal{I}$ such that  $f(\mathbf{s})= \mathbf{0}$.
       \end{teo}

It is clear that when we define $f$ to try to apply Poincar\'{e}-Miranda
theorem, the order of the components matters. So, sometimes to be
under the hypotheses of the theorem it is better to consider
$f=(f_{\sigma_1},f_{\sigma_2},\ldots,f_{\sigma_n})$ for some
permutation $\sigma.$ In fact, more in general, it is convenient to
apply the theorem to $A (f(\mathbf{x}))^t,$ where $A$ is a suitable
$n\times n$ invertible matrix. When $f$ is differentiable, as we
will see it is useful to chose $A=(\mathrm{D}f(\widehat
{\mathbf{s}}))^{-1},$ where $\widehat{\mathbf{s}}\in\Q^n$ is a
numerical approximation of a zero of $f$ in $\mathcal{I}.$

 If we succeed in proving that there is at least a solution in each box,
  its uniqueness is given by the fact that each of the intervals $ I_{j,m} $ contains
  only a single solution of each polynomial $q_j$. Otherwise we can refine
   boxes, taking them with smaller size,  and then repeating the computations in Step 1.

\section{An overview of the dynamics of $G$}\label{S-Landen}

In this section we briefly summarize the known results on the
dynamics of the map $G$ and we characterize their invariant sets. We
mainly follow the steps in \cite{chamb}. The rational integral
\eqref{eq6} is well-defined and convergent if $P(x)=x^3+ax^2+bx+1$
has not real positive roots. To study the number of real roots of
$P$ when $a$ and $b$ vary,  we consider
$$
  R(a,b):=\mathrm{Res}(P,P';x)=-\Delta_x(P)=-{a}^{2}{b}^{2}+4\,{a}^{3}+4\,
  {b}^{3}-18\,ab+27,
$$
\noindent where $\Delta_x$ is the discriminant. The curve $R(a,b)=0$
is known as the \emph{resolvent} one, and after removing the point
$(-1,-1)$ it is invariant by $G$ because
 \begin{equation}\label{E-Rinv}
  R(G(a,b))=\frac{(a-b)^2}{(a+b+2)^4}R(a,b).
 \end{equation}
 The curve has two connected components
 $L_1$ i $L_2$ (see Figure 1 (a)). Note that
 the fixed point $(3,3)$ is the cusp  of $L_1$.
 \begin{center}
  \includegraphics[scale=0.3]{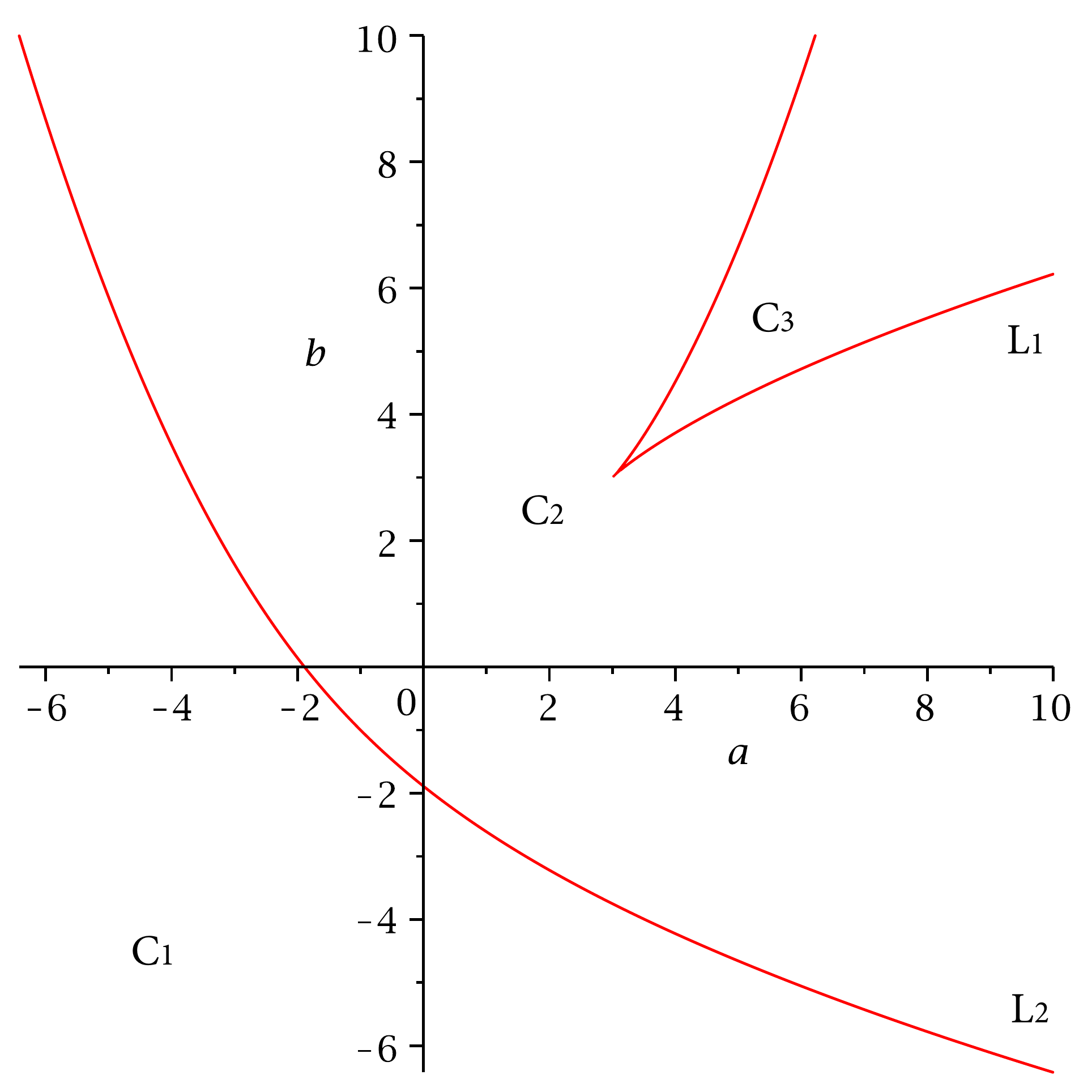} \hspace{1cm} \includegraphics[scale=0.55]{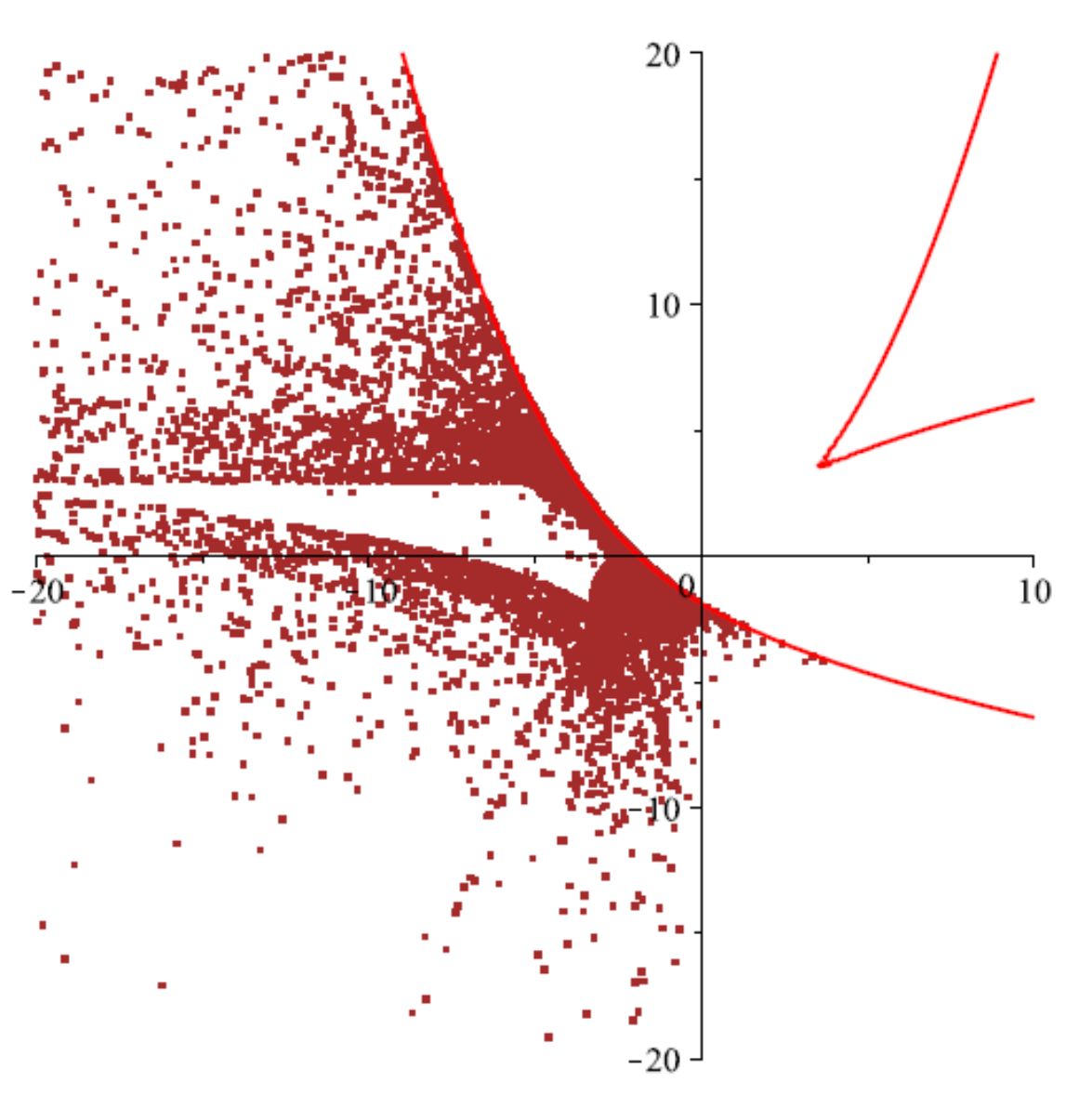}

(a) $\phantom{xxxxxxxxxxxxxxxxxxxxxxxxxxxxxxxxxxxxxxx} $ (b)

\smallskip

 Figure 1. (a) Connected components  $L_1$ and $L_2$ of the curve $R(a,b)=0$,
  and regions $C_1$, $C_2$ and $C_3$ of the plane. (b)  10000 iterates of an orbit with initial condition in $C_3$.
 \end{center}

The resolvent curve defines three open unbounded sets  $C_1$, $C_2$
and $C_3$ depicted in  Figure  1~(a). By studying the sign of the
discriminant of $P$, and by using the Descartes rule of signs, it is
straightforward to obtain that on $L_1\cup C_2\cup C_3$, all the
real roots of $P(x)$ are negative, so the integral  \eqref{eq6} is
convergent; and on $C_1\cup L_2$ there exists at least one positive
real root, so the integral diverges.

In \cite{chamb}, the authors proved that $G$ has only three fixed
points, namely $P_i$ for $i=1,2,3,$ which are described in Theorem
\ref{T-main} and given in Equation \eqref{E-PF}. One of them, the
point  $P_1=(3,3)$, is a super-attracting one, i.e. both eigenvalues
of the jacobian matrix are $0$. Their main result states that
\emph{the basin of attraction of the fixed point $P_1$ for the map
$G$, is the region of the $(a,b)$-plane where the
  integral \eqref{eq6} converge}. As a consequence, the basin of attraction  of $P_1$ is $L_1\cup C_2\cup C_3$.

On the other hand, the connected component
$L_2':=L_2\setminus\{(-1,-1)\}$ is \emph{positively} invariant (as
we will see, there are points on $C_1$ which are mapped into $L_2$).
On $L_2$ there is one fixed point, $P_2$, which is a saddle. In
Proposition \ref{P-atraccio-g} we prove that any orbit with initial
condition on $L_2'$ converges to $P_2$.

In summary, the dynamics of $G$ on the invariant sets
 $\mathcal{A}:=L_1\cup C_2\cup C_3$   and  $L_2$ is known and simple. However, there is a poor knowledge of
 the dynamics of $G$ in the set
$\mathcal{B}:=C_1\setminus \mathcal{F},$ where $
\mathcal{F}=\{(a,b)\in\R^2:\,\exists\, n\geq 0\,:\, G^n(a,b)\in
\{a+b+2=0\}\}$, is the forbidden set of $G$.

In \cite[Conj. 15.6.3]{llibre-Moll}, V.~Moll established the
following conjecture about the dynamics of~$G$ in~$\mathcal{B}$:
\emph{``The orbit of any point below the resolvent curve is dense in
the open region below this curve.''}

 Of course one has to exclude from this conjecture
the third fixed point $P_3$ which is in the set $C_1$, and the
points  in the forbidden set $\mathcal{F}$.  In Theorem
\ref{T-main}, we prove the existence of $3$-periodic points in
$C_1,$ result that disproves the conjecture. In fact it is not
difficult to find \emph{numerically} these orbits as well as other
periodic points, however to prove the existence of $3$-periodic
points is far from being trivial, and it is the main objective of
this paper.

 In Section \ref{S-Num-An} we present an analytic-numeric study that evidences the existence of points in the
 unstable manifold of $P_2$ which belongs to its stable
 set, i.e.  \emph{homoclinic points}. In the case of
  diffeomorphisms, by the Smale-Birkhoff homoclinic theorem, the
   existence of such points  implies the existence of a hyperbolic
   invariant set on which the dynamics is equivalent to a subshift of
    finite type, see \cite{GuH}. Similar results are developed in \cite{Gard,GarSus11} in the non-invertible setting.
  We also give evidences of the existence of points in the
  unstable manifold that also belong to the forbidden set~$\mathcal{F}$.

\section{Proof of Theorem \ref{T-main}}\label{S-fixos-i-periodics}

In this section we prove Theorem \ref{T-main}. We split the proof in
two subsections. The first one dedicated to the fixed and 2-periodic
points, and the second one to study the 3-periodic points.

\subsection{Fixed and $2$-periodic points}

\begin{proof}[Proof of statements (a) and (b)] (a) Following \cite{chamb}
 we consider the equations given by  $G(a,b)=(a,b)$, and we introduce the auxiliary variable $m^3=a+b+2$, obtaining
\begin{equation}\label{eq:pfixos}
\begin{cases}
  \begin{array}{ll}
  d_1(a,b,m)&:=m^3-a-b-2=0,\\
  d_2(a,b,m)&:=-am^4+ab+5a+5b+9=0,\\
  d_3(a,b,m)&:=-bm^2+a+b+6=0.
 \end{array}
 \end{cases}
 \end{equation}
Isolating $a$ and $b$ from the first and third equations, and
substituting the obtained expressions in the second one we get:
 \begin{align}\label{eq:fixos}
  d_4(m):=- ( m-2  )   ( {m}^{2}-m+1  )   ( {m}^{2}+m+2
          )   ( {m}^{3}+{m}^{2}-m-2  ) ( {m}^{3}+{m}^{2}+m +2  )=0.
 \end{align}
 The only real roots of the above equation
 are
{\scriptsize
\begin{align}\label{eq:fixos2}
m_1=2,\quad m_2=\frac{1}{6}\,\sqrt [3]{A}+\frac{8}{3}\,{\dfrac
{1}{\sqrt [3]
      {A}}}-\frac{1}{3}\simeq 1.20557,\quad
  m_3=-\frac{1}{6}\,\sqrt [3]{B}+\frac{4}{3}\,{\dfrac {1}
      {\sqrt [3]{B}}}-\frac{1}{3}\simeq -1.35321,
 \end{align}}
where $A=172+12\,\sqrt {177}$ and $B=188+12\,\sqrt {249}$. From
these values and \eqref{eq:pfixos} we obtain
 {\scriptsize \begin{align}\label{E-PF}
    P_1&=(3,3), \nonumber\\
    P_2&=\Bigg( {\frac
       {-43+3\sqrt {177}}{384}}  A^{2/3}-\frac{1}{6} A^{1/3}-\frac{8}{3},
       {\frac{13-\sqrt {177}}{48}}  A^{2/3}+\frac{7+\sqrt {177}}{48}  A^{1/3}
       +\frac{4}{3} \Bigg)\simeq(-4.20557,  3.95774), \\
    P_3&=\Bigg( {\frac{-21+\sqrt {249}}{96}}  B^{2/3}+\frac{15-\sqrt {249}}{12} B^{1/3} -2,{\frac{17-\sqrt {249}}{48}}  B ^{2/3}
       +{\frac{-13+\sqrt {249}}{24}} B^{1/3}-\frac{4}{3} \Bigg)\simeq(-5.30914,
       0.83118)\nonumber.
  \end{align}}
A  straightforward computation of the differential matrix at these
points  give that the points are, respectively, a super-attractor
(null eigenvalues), an oscillatory saddle, and an unstable focus.
Moreover $P_2$ is in $L_2$ and $P_3$ is in $C_1.$

\medskip

(b) Again, following \cite{chamb}, we consider $c$ and $d$ such that
 $G(a,b)=(c,d)$ and
 $G(c,d)=(a,b)$. By introducing the two auxiliary variables
 $m$ and  $n$ such that
 $m^3=a+b+2$ and $n^3=c+d+2$, we get:
$$\begin{cases}
\begin{array}{ll}
  d_1&:=m^3-a-b-2=0,\quad
  d_2:=n^3-c-d-2=0,\\
  d_3&:=-cm^4+ab+5a+5b+9=0,\quad
  d_4:=-dm^2+a+b+6=0,\\
  d_5&:=-an^4+cd+5c+5d+9=0,\quad
  d_6:=-bn^2+c+d+6=0.
 \end{array}
  \end{cases}
 $$
Solving
 $\{d_1=0,d_2=0, d_4=0, d_6=0\}$ we obtain
$$
  a={\dfrac {{m}^{3}{n}^{2}-{n}^{3}-2\,{n}^{2}-4}{{n}^{2}}},\quad
  b=\dfrac {{n}^{3}+4}{{n}^{2}},\quad
  c=\dfrac {{m}^{2}{n}^{3}-{m}^{3}-2\,{m}^{2}-4}{{m}^{2}},\quad
  d=\dfrac {{m}^{3}+4}{{m}^{2}}.
$$
By substituting the above result in  $d_3$ and
 $d_5$ we reach the following system, which plays the role of system \eqref{desc1} in our methodology:
\begin{equation}\label{E-2pd7d8}
\begin{cases}\begin{array}{rl}
 d_7(m,n):=&-{m}^{4}{n}^{7}+{m}^{5}{n}^{4}+2 {m}^{4}{n}^{4}
          +{m}^{3}{n}^{5}+5 {m}^{3}{n}^{4}+4 {m}^{2}{n}^{4}-{n}^{6}\\
          &+4 {m}^{3}{n}^{2}-2 {n}^{5}-{n}^{4}-8 {n}^{3}-8 {n}^{2}
          -16=0,\\
  d_8(m,n):=& d_7(n,m)=0.
\end{array}\end{cases}
\end{equation}

Now we consider the polynomial
\begin{align*}
  d_9(m):=&\mathrm{Res}(d_{7}(m,n),d_{8}(m,n);n)\\
  =&{m}^{4} \left( m-2 \right) \left( m+1 \right) ^{2}\left( {m}^{3}+{m}^{2}+m+2 \right)
\left( {m}^{3}+{m}^ {2}-m-2 \right)P(m),
\end{align*}
where $P$ is a polynomial of degree 56, without real roots. This is
proved by using the Sturm's method and can also be done, for
instance, by using the command \texttt{realroot} of the computer
algebra system Maple. Similarly,
$d_{10}(n):=\mathrm{Res}(d_{7}(m,n),d_{8}(m,n);m)$. As a consequence
of the symmetry we get
 that $d_{10}(n)=-d_{9}(n)$ and
  system $\{d_9(m)=0,\, d_9(n)=0\}$ plays  the role of system~\eqref{desc2}
 in our methodology. Hence, the only
 non-zero
 reals roots of $d_9$ are $-1,2, m_1$ and $m_2,$ where these values
 correspond to the ones associated with the fixed points, because
  the two degree~3 factors coincide with the ones given in
 \eqref{eq:fixos}.

Hence the 2-periodic points are included in the set with 16 elements
$\{-1,2,m_1,m_2\}^2.$ In this particular case, because the real
solutions of the uncoupled system are explicit, in Steps 3 and 4 of
our approach we have not boxes but points, and the problem is much
easier.  It is not difficult to check that in this set of points the
only solutions of \eqref{E-2pd7d8} are $(1,1), (m_1,m_1)$ and
$(m_2,m_2)$ which correspond to the fixed points of $G.$ In
consequence there are not points of minimal period $2$
for~$G$~\end{proof}

\subsection{Proof of Theorem \ref{T-main} (c): $3$-periodic points}

We
 need some preliminary results. Proceeding as in the previous
  cases we look for $a$, $b$, $c$, $d$, $e$, and $f\in\R$, such that $G(a,b)=(c,d)$, $G(c,d)=(e,f)$ and $G(e,f)=(a,b)$,
that is
\begin{align*}
  &\dfrac{5a+5b+ab+9}{(a+b+2)^{4/3}}=c,
  &&\dfrac{a+b+6}{(a+b+2)^{2/3}}=d,
  &&\dfrac{5c+5d+cd+9}{(c+d+2)^{4/3}}=e,\\[0.1cm]
  &\dfrac{c+d+6}{(c+d+2)^{2/3}}=f,
&&\dfrac{5e+5f+ef+9}{(e+f+2)^{4/3}}=a,
  &&\dfrac{e+f+6}{(e+f+2)^{2/3}}=b.
\end{align*}
We introduce the auxiliary variables
 $m$, $n$ and $r$, such that
 $m^3=a+b+2$, $n^3=c+d+2$  and $r^3=e+f+2$.
Using this notation we get
$$\begin{cases}
\begin{array}{lll}
  d_1:=m^3-a-b-2=0,& d_4:=-cm^4+ab+5a+5b+9=0,& d_5:=-dm^2+a+b+6=0,\\
  d_2:=n^3-c-d-2=0,& d_6:=-en^4+cd+5c+5d+9=0,& d_7:=-fn^2+c+d+6=0,\\
  d_3:=r^3-e-f-2=0,& d_8:=-ar^4+ef+5e+5f+9=0,&d_9:=-br^2+e+f+6=0.
\end{array}
\end{cases}
$$

First we solve the system
 $\{d_1=0,d_2=0,d_3=0,d_5=0,d_7=0,d_9=0\}$ obtaining:
\begin{align}\label{E-abcdef}
  a&=\dfrac{{m}^{3}{r}^{2}-{r}^{3}-2\,{r}^{2}-4}{{r}^{2}},\quad
  b=\dfrac{{r}^{3}+4}{{r}^{2}},\quad
  c=\dfrac{{m}^{2}{n}^{3}-{m}^{3}-2\,{m}^{2}-4}{{m}^{2}},\nonumber \\
  d&=\dfrac{{m}^{3}+4}{{m}^{2}}, \quad
  e=\dfrac{{n}^{2}{r}^{3}-{n}^{3}-2\,{n}^{2}-4}{{n}^{2}}, \quad
  f=\dfrac{{n}^{3}+4}{{n}^{2}}.
\end{align}
Substituting this result in the expressions $d_4,$ $d_6,$ and $d_8$,
we obtain the equations
 \begin{equation}\label{E-d10d11d12}
\begin{cases}
 \begin{array}{ll}
  d_{10}(m,n,r):=&-{m}^{4}{n}^{3}{r}^{4}+{m}^{5}{r}^{4}
    +2\,{m}^{4}{r}^{4}+{m}^{3}{r}^{5}+5\,{m}^{3}{r}^{4}
    +4\,{m}^{2}{r}^{4}-{r}^{6}+4\,{m}^{3}{r}^{2}\\
    &-2\,{r}^{5}-{r}^{4}-8\,{r}^{3}-8\,{r}^{2}-16=0,\\
  d_{11}(m,n,r):=&d_{10}(n,r,m)=0,\\
  d_{12}(m,n,r):=&d_{10}(r,m,n)=0.
\end{array} \end{cases}
 \end{equation}

From the equations \eqref{E-abcdef}, if $(m,n,r)$ is a real solution
of \eqref{E-d10d11d12} such that $m \cdot n\cdot  r\neq 0$,
 there exists either an orbit with minimal period
 $3$ given by  \eqref{E-abcdef} or a fixed point or $G$. Moreover,
 if $(m_0,n_0,r_0)$ is a solution of system  \eqref{E-d10d11d12}, then so are
 $(n_0,r_0,m_0)$ and $(r_0,m_0,n_0)$. As a consequence, we obtain

\begin{lem}\label{O-1}
Any $3$-periodic orbit of  $G$,
  $\{(a,b); (c,d); (e,f)\}$ with associated parameters $m,n$ and $r$,
   is in correspondence, via \eqref{E-abcdef}, with the solutions $(m,n,r)$, $(n,r,m)$ and $(r,m,n)$ of  the system~\eqref{E-d10d11d12}.
  \end{lem}

The forthcoming  Lemma \ref{L-lem2} gives a first characterization
of the locus where the solutions of system \eqref{E-d10d11d12} are
located. Prior to state this result we introduce the following
auxiliary polynomials
 \begin{align*}
   d_{13}(n,r)&:=\mathrm{Res}(d_{10}(m,n,r),d_{12}(m,n,r);m),
        \, \mbox{with degree $37$ in $n$ and degree $37$ in $r$,}\\
   d_{14}(n,r)&:=\mathrm{Res}(d_{11}(m,n,r),d_{12}(m,n,r);m),
        \, \mbox{with degree $47$ in $n$ and degree $37$ in $r$.}\
 \end{align*}
We apply the resultant once again to obtain the polynomials $
   d_{15}(n):=\mathrm{Res}(d_{13}(n,r),d_{14}(n,r);r),$ and
$   d_{16}(r):=\mathrm{Res}(d_{13}(n,r),d_{14}(n,r);n),
$
where $\mathrm{deg}_n(d_{15}(n))=2521$ and $\mathrm{deg}_r(d_{16}(r))=1985$.
Finally, we introduce the polynomial
\begin{equation}\label{eq:d17}
   d_{17}(n):=\gcd\left(d_{15}(n),d_{16}(n)\right)/n^{716}.
\end{equation}
 This polynomial has degree $371$, and using once more Sturm's method we get that
 it has exactly $16$ different
real  non-zero roots.

\begin{lem}\label{L-lem2}
Let  $I_i,$ with $i=1,\ldots, 16,$ be  disjoint intervals, each one
of them containing a unique real root of $d_{17}$. Then, any real
solution  $(m,n,r)$ of system  \eqref{E-d10d11d12} is contained in
one of the $16^3$ sets
\begin{equation}\label{E-caixesR3}
\mathcal{I}_{i,j,k}:=I_i\times I_j\times I_k,\quad i,j,k\in\{1,\ldots,16\}.
\end{equation}
 \end{lem}
\begin{proof}
Let $(m_0,n_0,r_0)$ be a real solution of \eqref{E-d10d11d12}.  We
want to show that it is also a solution of $\{d_{17}(m)=0,
d_{17}(n)=0, d_{17}(r)=0\}$. Observe that by construction, $n_0$
must be a root of $d_{15}$. From  Lemma \ref{O-1},  we have that
$(m_0,n_0)$ must be also a zero of $d_{13}$ and $d_{14}$. Hence
$n_0$ must be also a zero of $d_{16}$, and therefore of
$\gcd\left(d_{15}(n),d_{16}(n)\right)$ which is a polynomial of
degree $1087$ with the factor $n^{716}$. Since we are interested in
its non-zero roots, we remove this factor, obtaining that $n_0$ must
be a root of $d_{17}$. By using an analogous argument and Lemma
\ref{O-1} again, we can see that $m_0$ and $r_0$ are also roots of
$d_{17}$.

Since $d_{17}$ has $16$ different  real roots, any solution
$(m,n,r)$ of system \eqref{E-d10d11d12} must be contained in a box
of the form \eqref{E-caixesR3}, and each box contains at most one
solution.
\end{proof}

Now we can prove statement (c) of Theorem \ref{T-main}. We follow
the steps explained in Section  \ref{Ss-metode-general}:

\smallskip

 \noindent \textbf{Step 1:} Recall that $d_{17}$ has  16 non-zero real roots. Two of them are $n=-1$, $n=2.$
 Although two more explicit roots are
  $n_i=m_i, i=1,2$ given in \eqref{eq:fixos2}, we prefer to take 14 intervals with rational ends and length smaller than $10^{-20},$
 $I_i, i=1,2,\ldots,14,$ each one of them containing a unique root.
 We consider:

{\tiny
\begin{align*}
&I_1= \left[
-{\frac{4308988841618670568853}{147573952589676412928}},-{\frac{
34471910732949364550823}{1180591620717411303424}} \right], && I_2=
\left[ -{\frac{
34411805733101949308435}{1180591620717411303424}},-{\frac{
17205902866550974654217}{590295810358705651712}}  \right],\\[0.1cm]
 &I_3= \left[ -{\frac{
9138398550509024508051}{1180591620717411303424}},-{\frac{
4569199275254512254025}{590295810358705651712}}  \right], &&
 I_4= \left[ -{\frac{
4416518740855918762195}{590295810358705651712}},-{\frac{
8833037481711837524389}{1180591620717411303424}}  \right],\\[0.1cm]
&I_5= \left[ -{\frac{
994661336537171251825}{295147905179352825856}},-{\frac{
3978645346148685007299}{1180591620717411303424}}  \right], && I_6=
\left[ -{\frac{
3977374161031280580629}{1180591620717411303424}},-{\frac{
994343540257820145157}{295147905179352825856}}  \right],\\[0.1cm]
 &I_7= \left[-{\frac{
197879469664271669175}{73786976294838206464}},-{\frac{
3166071514628346706799}{1180591620717411303424}}   \right], &&
 I_8= \left[ -{\frac{
3144313156826151948503}{1180591620717411303424}},-{\frac{
1572156578413075974251}{590295810358705651712}}  \right],\\[0.1cm]
 &I_9= \left[ -{\frac{
399397086201257638833}{295147905179352825856}},-{\frac{
1597588344805030555331}{1180591620717411303424}}  \right], &&
 I_{10}= \left[-{\frac{
1053526769518098399097}{4722366482869645213696}},-{\frac{
131690846189762299887}{590295810358705651712}}  \right],\\[0.1cm]
&I_{11}= \left[ -{\frac{
1064910654630154190265}{9444732965739290427392}},-{\frac{
133113831828769273783}{1180591620717411303424}}  \right],&&
 I_{12}= \left[ {\frac{
1065572958580542810237}{9444732965739290427392}},{\frac{
532786479290271405119}{4722366482869645213696}}  \right],\\[0.1cm]
&I_{13}= \left[ {\frac{
128535594827653577343}{590295810358705651712}},{\frac{
1028284758621228618745}{4722366482869645213696}}  \right],
 &&I_{14}= \left[ {\frac{
177910645965499912685}{147573952589676412928}},{\frac{
1423285167723999301481}{1180591620717411303424}}  \right].\\
\end{align*}
}
 We  also introduce the degenerate intervals $I_{15}=[-1,1]$ and
$I_{16}=[2,2]$ containing the exact roots $n=-1$ and $n=2$. By Lemma
\ref{L-lem2}, all the real solutions of system  \eqref{E-d10d11d12}
are contained in one of  the $16^3$  boxes \eqref{E-caixesR3}, where
we also call boxes the ones with some degenerate interval. Recall
that if a box $\mathcal{I}_{i,j,k}$ contains a solution of system
\eqref{E-d10d11d12}, then this solution is unique.

\smallskip

\noindent \textbf{Step 2:} We apply the discard procedure  to
$d_{10}$ and  the $4096$ boxes  of the form \eqref{E-caixesR3} given
by the intervals computed before. In  Lemma \ref{L-minmaj-nou-n} we
use the value $\xi=30$ and consider the polynomial
$P(m,n,r)=d_{10}(m-\xi,n-\xi,r-\xi),$ which has $224$ monomials. The
procedure implemented in Maple v.17 took 5.61s of real time in an
Intel i7-3770-3.4GHz CPU to discard $4080$ boxes. The code is given
in \cite[Chap. 5]{Llor}. In short, we obtain that each solution of
system \eqref{E-d10d11d12} must be contained in one of the following
$16$ non-discarded boxes
$$
       \begin{array}{| l | l | l | l |  l | l | l | l | }
     \hline
     \mathcal{I}_{1,5,11} & \mathcal{I}_{2,6,12} & \mathcal{I}_{3,7,13} & \mathcal{I}_{4,8,10}  &
     \mathcal{I}_{5,11,1} & \mathcal{I}_{6,12,2} & \mathcal{I}_{7,13,3} & \mathcal{I}_{8,10,4}  \\
     \hline
       \mathcal{I}_{9,9,9} & \mathcal{I}_{10,4,8} & \mathcal{I}_{11,1,5} & \mathcal{I}_{12,2,6} &
      \mathcal{I}_{13,3,7} & \mathcal{I}_{14,14,14} & \mathcal{I}_{16,16,15} & \mathcal{I}_{16,16,16}\\
     \hline
    \end{array}
   $$
Observe that the degenerated box $\mathcal{I}_{16,16,15}$, which
corresponds  with $(m,n,r)=(2,2,-1)$, must also be discarded because
 $d_{10}(2,2,-1)=0$, but $d_{11}(2,2,-1)=2304$.

\smallskip

\noindent \textbf{Step 3:} Following similar arguments that in the
proof of statement (b) we can discard boxes
$\mathcal{I}_{16,16,16},$ $\mathcal{I}_{14,14,14}$ and
$\mathcal{I}_{9,9,9}$ because they correspond to the fixed points
$P_1,P_2$ and   $P_3,$ respectively.

\smallskip

\noindent \textbf{Step 4:} We have obtained
 $12$ non-discarded boxes  that, from Lemma \ref{O-1}, if they correspond to periodic points of minimum period $3$,
 they would contain the parameters $(m,n,r)$ corresponding to the periodic points
according to the following groupings:
%
%
%

\begin{equation}\label{E-supercaixes}
 \begin{array}{cc}
\mathcal{O}_1\subset    
     \mathcal{I}_{1,5,11}\cup \mathcal{I}_{5,11,1}\cup \mathcal{I}_{11,1,5},&

\mathcal{O}_2\subset   
      \mathcal{I}_{2,6,12}\cup \mathcal{I}_{6,12,2}\cup  \mathcal{I}_{12,2,6},\\  

{}\\
\mathcal{O}_3\subset 
       \mathcal{I}_{3,7,13}\cup \mathcal{I}_{7,13,3}\cup \mathcal{I}_{13,3,7}, &

 \mathcal{O}_4\subset \mathcal{I}_{4,8,14}\cup \mathcal{I}_{8,14,4}\cup \mathcal{I}_{14,4,8}.
\end{array}
\end{equation}

We will prove that the above $12$ boxes indeed contain a solution
of system   \eqref{E-d10d11d12}, which will be unique as reasoned
above. To do this, we will apply the Poincar\'e-Miranda theorem
(Theorem \ref{T-PM-n}). Again by Lemma \ref{O-1} we only need to
prove that there is a solution of the system \eqref{E-d10d11d12} in
the boxes:
  $\mathcal{I}_{1,5,11}$,  $\mathcal{I}_{2,6,12}$,  $\mathcal{I}_{3,7,13}$, and  $\mathcal{I}_{4,8,14}$.
  For reasons of space we only give details for the first box.

We consider the polynomial map
$f(m,n,r):=\left(d_{10}(m,n,r),d_{11}(m,n,r),d_{12}(m,n,r)\right).$
We  denote the ends of the intervals $I_1$, $I_5$ and $I_{11}$
respectively: $ [\underline m,\overline m]:=I_1,\,[\underline n
,\overline n]:=I_5,\,[\underline r,\overline  r]:=I_{11}.$ Consider
also
 the middle point of    $\mathcal{I}_{1,5,11},$ $ \widehat
{p}=(\widehat {m},\widehat {n},\widehat {r})=\left(\big(\underline m
+\overline m\big)/2,\big(\underline n+\overline
n\big)/2,\big(\underline r+\overline r\big)/2\right). $

 \begin{center}
 \begin{minipage}{0.5\textwidth}
  \includegraphics[scale=0.50]{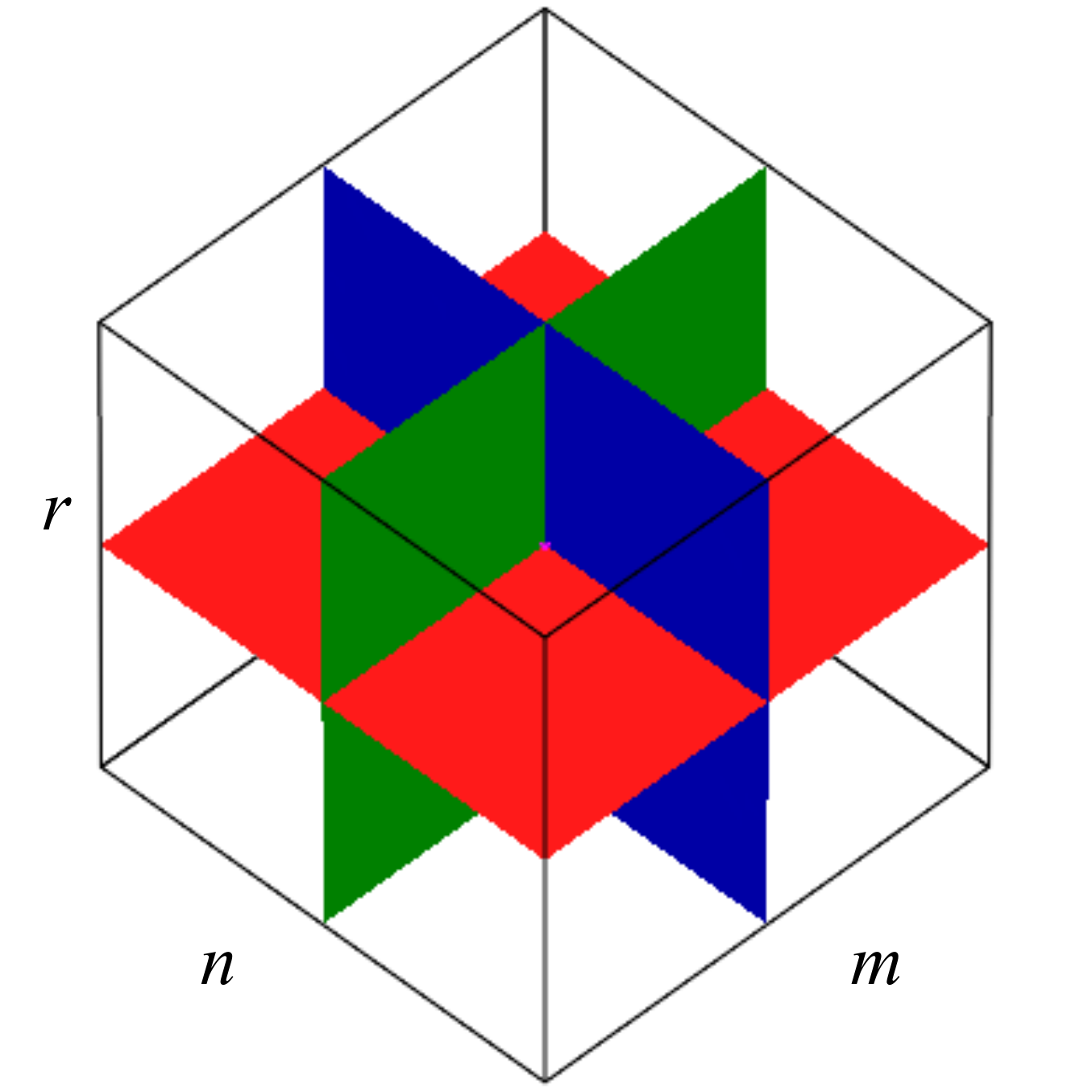}
  \end{minipage}
\begin{minipage}{0.49\textwidth}
 Figure 2. The surfaces $g_1(m,n,r)=0$, $g_2(m,n,r)=0$  and $g_3(m,n,r)=0$ in blue, green and red, respectively, in the box
 $[\widehat {m}-\varepsilon,\widehat {m}+\varepsilon]
 \times[\widehat {n}-\varepsilon,\widehat {n}+\varepsilon]\times[\widehat {r}-\varepsilon,\widehat {r}+\varepsilon],$  where $\varepsilon=10^{-10}$.
\end{minipage}
\end{center}

The hypothesis  of Poincar\'e-Miranda theorem for $f$ using the box
$\mathcal{I}_{1,5,11}$ are not satisfied: for instance, at the
points $(\underline m,\widehat {n},\widehat {r})$ and $(\overline
m,\widehat {n},\widehat {r})$ none of the functions $d_{10}, d_{11}$
and $d_{12}$ changes  sign. So in order to rectify the level 0
surfaces of the components of $f$, we consider the new function
$$
g(m,n,r)=\left(g_1(m,n,r),g_2(m,n,r),g_3(m,n,r)\right):=(\mathrm{D}f(\widehat
{p}))^{-1} (f(m,n,r))^t.
$$
 We omit here the expressions of
$(\mathrm{D}f(\widehat {p}))^{-1}$ and $g$ since they involve huge
rational numbers with numerators and denominators with hundreds of
digits. Notice that since
$\operatorname{det}\left(\mathrm{D}f(\widehat {p})\right)\ne0$
 the point $(m_0,n_0,r_0)$ is a zero
of  $g$ if and only if it is a zero of  $f$.

Observe that  $ g(m,n,r)=g(\widehat {p})+(m-\widehat {m},n-\widehat
{n},r-\widehat {r})+O(||(m-\widehat {m},n-\widehat {n},r-\widehat
{r})||^2)$. Since $g(\widehat {p})\simeq0,$ near $\widehat p$ it
holds that  $ g(m,n,r)\simeq (m-\widehat {m},n-\widehat
{n},r-\widehat {r})$ and so, a small enough box centered at
$\widehat p$ should be under the hypotheses of Poincar\'{e}-Miranda
theorem, see Figure 2. Now we will check that, indeed, this is  the
situation for the function $g$ in the box $\mathcal{I}_{1,5,11}$.

 In order to prove that the components of the function $g$ have
  no roots, and alternate signs at the faces of  $\mathcal{I}_{1,5,11}$
   we will apply repeatedly the following technical result, that is a simplified
   version adapted to our interests
   of a result given in
   \cite{JD2}:

  \begin{lem}\label{zeros2}
   Let $G_\alpha(x)=g_n(\alpha) x^n+g_{n-1}(\alpha) x^{n-1}+\cdots+g_1(\alpha)
   x+g_0(\alpha)$ be
   a family of real polynomials that depend continuously on a real parameter
 $\alpha\in\Lambda=[\alpha_1,\alpha_2]\subset\mathbb{R}$. Fix $J=[a,b]\subset\R$ and assume that:
   \begin{enumerate}[(i)]
    \item There exists $\alpha_0\in \Lambda$ such that  $G_{\alpha_0}(x)$ has no  real roots in  $J$.
    \item For all $\alpha\in \Lambda$,  $G_\alpha(a)\cdot G_\alpha(b)\cdot \Delta_x(G_\alpha)\neq 0,$ where
    $\Delta_x(G_\alpha)$
    is the discriminant of $G_\alpha$ with respect to $x.$
  \end{enumerate}
   Then for all $\alpha\in \Lambda$, $G_\alpha(x)$  has no  real roots in $ J$.
  \end{lem}

We will prove that the first component of $g$ has no roots, and
alternates signs at the faces $m=\underline m$ and $m=\overline m$
of the box $\mathcal{I}_{1,5,11}$. Consider the function $
G_n(r)=g_1(\underline m,n,r)\cdot g_1(\overline m,n,r). $ We will
prove that $G_n(r)<0$ for all $(n,r)\in I_5\times I_{11}$ using
Lemma~\ref{zeros2} with $\Lambda=I_5$ and $J=I_{11}.$

By the Sturm's method it can be seen that the polynomial
$G_{\widehat {n}}(r)$ has only 6 different real roots and that none
of them is in the interval $I_{11}.$ Hence the hypothesis $(i)$ is
satisfied. Moreover $G_{\widehat {n}}(r)$ restricted to $I_{11}$ is
negative.

Proceeding in an analogous way, we obtain that $G_n(\underline
r)\cdot G_n(\overline r)$ has only 4 different real roots and  none
of them belongs to $I_{5}$. Hence $G_n(\underline r)\cdot
G_n(\overline r)\neq 0$ for all $n\in I_5$.  We also check that the
discriminant $\Delta_{r}(G_n(r))$, which is a polynomial of degree
192 in $n$, has  $37$ different  real roots. Again, we  prove that
they are not in $I_5$ and so, we are under the hypothesis $(ii)$ of
Lemma \ref{zeros2}. Hence by  this lemma   we get that $G_n(r)<0$
for all $(n,r)\in I_5\times I_{11},$ as we wanted to prove.

Doing similar arguments and computations,  we obtain that the second
and third component of $g$ do not vanish, and alternate signs on the
faces  $n=\underline n$ and $n=\overline n$, and $r=\underline r$
and $r=\overline r$ of $\mathcal{I}_{1,5,11}$, respectively. See
\cite[Chap. 5]{Llor} for more details. Thus  $g(m,n,r)$ verifies the
hypothesis of the Poincar\'e-Miranda theorem in
$\mathcal{I}_{1,5,11}.$ Hence the function $g$, and therefore the
function $f$, have at least  one zero in this box, which is unique
by construction.

\subsection{Analytic location of the $3$-periodic points}\label{ss:location}

In this section we use that the parameters $m$, $n$ and $r$
associated to each periodic point are located in the $12$ boxes
given in \eqref{E-supercaixes}, to obtain an analytic location of
them in the $(a,b)$-plane.

\begin{lem}\label{L-acotacions-final}
Let $m\in[\underline m ,\overline m ]$ and $r\in[\underline r
,\overline r ]$, and the functions $a(m,r)$ and $b(r)$ given by
\eqref{E-abcdef}, then:
\begin{enumerate}
\item[(i)] If $0<\underline r   \leq \overline r   $ then
$\underline a    :=\underline m ^3-\overline r   -2-{4}/{\underline
r   ^2}\leq a(m,r)\leq \overline m ^3-\underline r -2-{4}/{\overline
r ^2}=:\overline a     $ and $\underline b      :=\underline r
+{4}/{\overline r ^2}\leq b(r)\leq \overline r +{4}/{\underline r
^2}=:\overline b $

\item[(ii)] If $\underline r   \leq \overline r   <0$ then
$ \underline a    :=\underline m ^3-\overline r   -2-{4}/{\overline
r   ^2}\leq a(m,r)\leq \overline m ^3-\underline r
-2-{4}/{\underline r ^2}=:\overline a     $ and $\underline b
:=\underline r +{4}/{\underline r   ^2}\leq b(r)\leq \overline r
+{4}/{\overline r ^2}=:\overline b      $.
\end{enumerate}
\end{lem}
\begin{proof}
(i) From \eqref{E-abcdef} we get that $a(m,r)=m^3-r-2-{4}/{r^2}$ and
$ b(r)=r+{4}/{r^2}.$ Notice that if $0<\underline r \leq r \leq
\overline r $ then $-\overline r \leq -r \leq -\underline r <0$ and
 $ -\frac{4}{\underline r ^2}\leq -\frac{4}{r^2}\leq
-\frac{4}{\overline r ^2}<0 .$ Moreover, $\underline m ^3\leq m^3
\leq \overline m ^3.$ By combining all these chains of inequalities
we get statement (i). The statement (ii) follows similarly.~\end{proof}

By using the inequalities in $(ii)$ of Lemma
\ref{L-acotacions-final}  we obtain, for example, that the
$3$-periodic point $(a,b)$ of $G$, associated to the parameters
$(m,n,r)\in \mathcal{I}_{1,5,11}$, satisfies $a\in[\underline a
,\overline a ]$ and $b\in[\underline b ,\overline b ]$ where

{\tiny \begin{align*}\underline a    &= -\frac {
1435686715756812113129131753291751212473714621389705932746390847605145815709035232062993533718832495489341
}{
56947609584619278435915236206283183709714097978506070511694763452312581699417401160811385506316156928
},\\
 \overline a     &=-\frac{
47044582301919219323098597682011430719430330620984084471100755414697990442772197382375529104298060913286874119
}{
1866059270868804515791575090678155019012542140400238364480469557193740712623709418953041434224970065510400
}
\end{align*}}
and {\tiny \begin{align*} \underline b      &={\frac {
3368785687756582636246263551756811406295236320753178521304454421527}{
10710654937528498667637446691242283113536911386660380934878003200}},\\
\overline b     & ={\frac {
6579659546399575461418490144259606329620802274396204966850496409}{
20919247924860348960190099800217926294342327605140376818548736}}.
\end{align*}}
By using the decimal approximation we get, {\scriptsize
\begin{align*}
&a\in[\underline a    ,\overline a     ]\simeq
[-25210.658115921519312682, -25210.658115921519312679],\\
&b\in[\underline b      ,\overline b      ]\simeq
[314.5265819322469464743, 314.5265819322469464749],
\end{align*}}
where we observe that $\max(\overline a     -\underline a
,\overline b      -\underline b      )\simeq 2.8\times 10^{-18}.$

 Applying  Lemma \ref{L-acotacions-final} to each of the $12$ boxes \eqref{E-supercaixes}, we obtain
rational bounds for the components of each periodic point of minimal
period $3$, that are summarized in the following tables, where only
the decimal expression of some significative digits is given. In all
the cases the maximum length of the interval localizing the
3-periodic points is smaller than $10^{-17},$ so the given expression of both ends of the intervals coincide.

{\scriptsize
\begin{center}
\begin{tabular}{|c|c|c|}
\hline
$\mathcal{O}_1$ & $a$ & $b$ \\
\hline
 & -25210.658115921519313 & 314.52658193224694647 \\
\hline
 & -11.080089229288244821& -29.194152462502174029 \\
\hline
 & 1.0164106270635353803& -3.0178440371837045505  \\
\hline
\end{tabular}
\begin{tabular}{|c|c|c|}
\hline
$\mathcal{O}_2$ & $a$ & $b$ \\
\hline
 & -25080.503857555317449&  314.36115078061939834\\
\hline
 & -11.094342178650567807& -29.143225143670723223 \\
\hline
 & 1.0179782228602330827&  -3.0165421366176918413 \\
\hline
\end{tabular}
\end{center}}

{\scriptsize
\begin{center}
\begin{tabular}{|c|c|c|}
\hline
$\mathcal{O}_3$ & $a$ & $b$ \\
\hline
 & -550.35997876621370288& 84.580855473468510676 \\
\hline
 &-13.613164340185764400 & -7.6737642167841728949 \\
\hline
 & 0.13590789992610542444&  -2.1255835876361107899 \\
\hline
\end{tabular}
\begin{tabular}{|c|c|c|}
\hline
$\mathcal{O}_4$ & $a$ & $b$ \\
\hline
 &-500.96942815695686889 & 80.145842594816842809 \\
\hline
 &-13.481597649423988848 &  -7.4104176831057891201 \\
\hline
 &  0.088325991394389446424&   -2.0994294342645985249\\
\hline
\end{tabular}
\end{center}}

\section{Dynamics associated to the saddle point $P_2$}
\label{S-Num-An}

In this section we study the invariant sets of the saddle point
$P_2$. First,  Proposition \ref{P-atraccio-g}  characterizes the
stable set of $P_2.$ We will also give numerical evidences of the
existence of homoclinic orbits, that is, initial conditions on the
local unstable manifold, whose orbit converges to $P_2$. Finally, we
will provide numerical evidences of the existence of points on the
local unstable manifold, whose orbits end in the non-definition set.

\subsection{The stable set of $P_2$}\label{SS-stab-set}

We denote the \emph{stable set} of the fixed point $P_2$ as
$W^s(P_2)=\{(a,b)\in \R^2:\, \lim_{n\to\infty} G^n(a,b)=P_2\}.$ This set is
not necessarily a manifold. For hyperbolic points, like $P_2$ there
is also the so called  \emph{local stable manifold}
$W^s_{\textrm{loc}}(P_2)\subset W^s(P_2),$ that is an actual
manifold and is only defined in a small neighborhood of the fixed
point. Our first result characterizes totally the stable set of
$P_2$ and its local stable manifold.

 \begin{propo}\label{P-atraccio-g} It holds that
 \[
W^s(P_2)=\big(L_2\cup\{(a,b)\in\R^2:\, \exists\, n\ge0: G^n(a,b)\in
R_1\cup C_1\}\big)\setminus\{(-1,-1)\},
 \]
where $R_1=\{a-b=0\}$. Moreover, $W^s_{\textrm{loc}}(P_2)$ is
contained in $L_2.$
 \end{propo}

 \begin{proof} Observe that Equation (\ref{E-Rinv}) implies that the only initial conditions
 mapped by $G$  to the resolvent curve are the points of the straight line
 $R_1,$ except $(-1,-1).$
  Hence, to prove the proposition it suffices to show that $L_2'=L_2\setminus\{(-1,-1)\}\subset W^s(P_2).$
Let us prove this inclusion.

   Recall that $L_2\subset\{R(a,b)=0\}.$ The
resolvent curve $R(a,b)=0$ is algebraic and has genus 0, so it
admits \emph{rational parametrizations}. For instance, if we define
 $P(t)=(P_1(t),P_2(t))=\left(\frac{t^3+4}{t^2},\frac{t^3+16}{4t}\right)$ it holds that
 $R(P_1(t),P_2(t))\equiv0.$ This parametrization has been already   was also used in \cite[Thms 3 and 4]{chamb}.
The component $L_2$ corresponds with
 $t\in(-\infty,0)$, and $L_1$ with
 $t\in(0,\infty)$. Some computations give
  $P^{-1}(a,b)={\frac{4\,({a}^{2}-3\,b)}{{a}^{2}b-4\,{b}^{2}+3\,a}}.$
Hence, to study the dynamics of $G$ on the component
$L_2$ we need to study the one-dimensional map
 $$
 g(t)=P^{-1}\circ G\circ P(t)=\sqrt[3]{4}\,\frac{t}{(t+2)^2}\,
 \left(\frac{(t^2+4)(t+2)^2}{t^2}\right)^{2/3}, \mbox{ for }t\in
 \mathcal{I}:=(-\infty,0)\setminus\{-2\},
 $$ see  also \cite[Thm 4]{chamb}.   Observe that $t=-2$
 corresponds with $(a,b)=(-1,-1)$ which belongs to the non-definition line $\{a+b+2=0\}$ and is excluded in our statement.  The  map $g(t)$ has a
 unique fixed point in $\mathcal{I}$
$$
p=-\frac{1}{3}\,{\frac{4\,C+\sqrt [3]{2}{C}^{2}+8\,{2}^{2/3}}{C}}\simeq -4.4111,\,\mbox{ where } C=\sqrt [3]{86+6\,\sqrt {177}}.
$$
Our objective is to prove  that this fixed point is a global attractor of  $g(t)$ in $\mathcal{I}$.

First we summarize some   features of $g(t)$ in $\mathcal{I}$ that
we will need (see Figure 3): (i) It has only two relative extremes
(maximum) in $\mathcal{I}$ given by
 $t=-4\mp 2\sqrt{3}$ ($t\simeq -7.4641$ and $t\simeq -0.5359$ respectively), and such that $g(-4\mp 2\sqrt{3})=-4$.
We denote $m:=-4- 2\sqrt{3}$. (ii) It holds that
$\lim\limits_{t\to-2^{\pm}}g(t)=\lim\limits_{t\to0^{-}}g(t)=-\infty$.
(iii) It  also holds that $\lim\limits_{t\to-\infty}g(t)=-\infty$.
(iv) For all $t\in(-\infty,p)$, we have $g(t)>t$.
 (v) The map $g$ has not  $2$-periodic points  as  a
consequence of  Theorem \ref{T-main} (b).

 \begin{center}
  \includegraphics[scale=0.42]{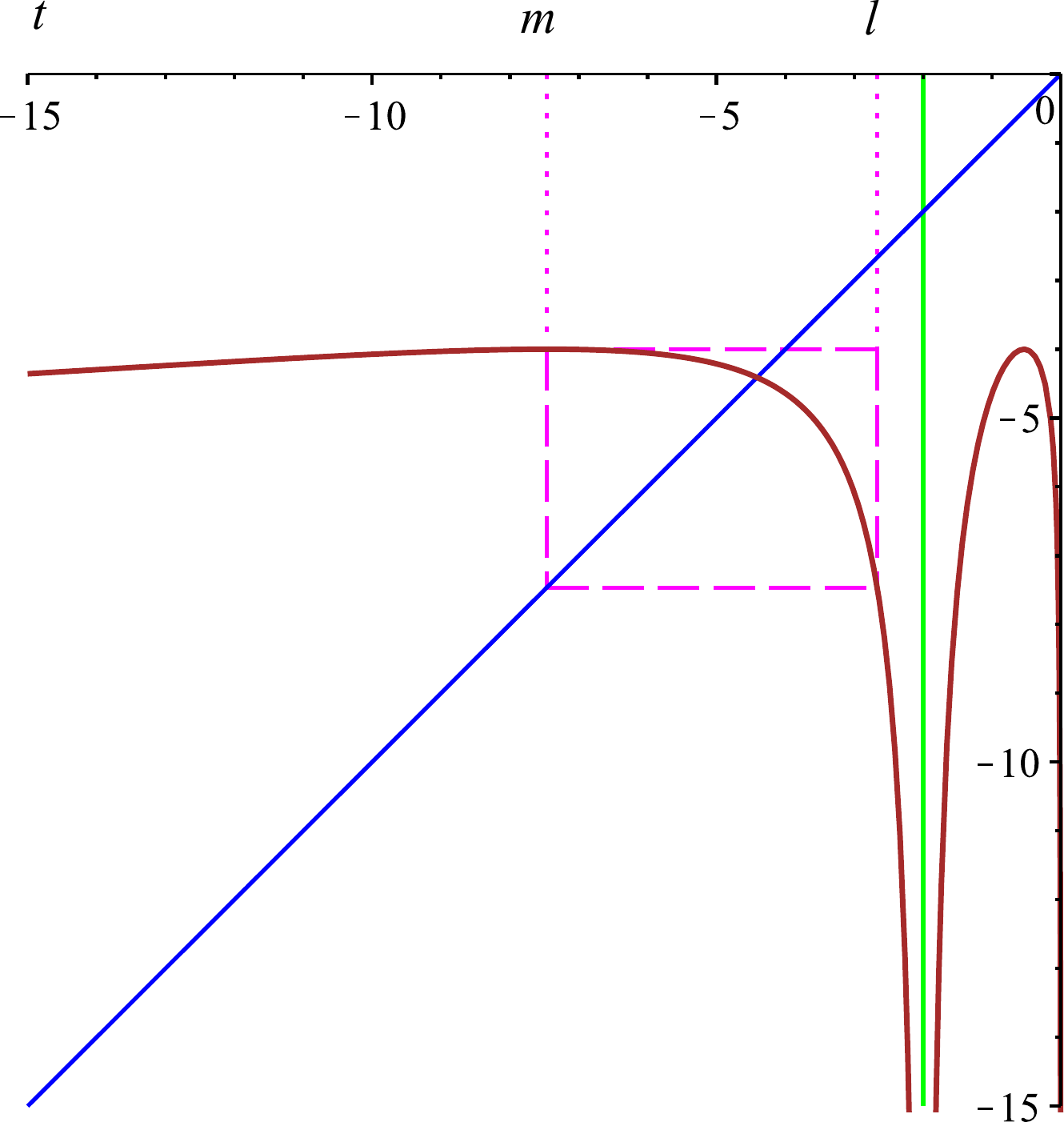}

 Figure 3. Graph of the function  $g(t)$ in  $\mathcal{I}$.
 \end{center}

The proof has three steps, namely (A)--(C):
 (A) From  the properties (i) and (ii), we conclude $g\left((-2,0)\right) = (-\infty,-4]$.

\noindent (B) Using (i) and (ii) again, we conclude that
 $g\left((-\infty,-2)\right) = (-\infty,-4]$,
 hence the interval $(-\infty,-2)$ is invariant by  $g$.
We will study the dynamics in this interval.

Let $\ell\in(p,-2)$ be the unique value in this interval such that  $g(\ell)=m$ ($\ell\simeq -2.6675$).
By using the monotony of $g$, the interval $[m,\ell]$ is invariant. Indeed, $g([m,\ell])=[g(\ell),g(m)]=[m,-4]\subset[m,\ell],$
see again Figure 3. Now we claim that \emph{for all $t\in (-\infty,m)\cup(\ell,-2)$
there exists $n>0$ such that $t_n=g^n(t)\in [m,\ell]$.}

Indeed, by the monotonicity of  $g$ in $(\ell,-2)$ we have that
$g((\ell,-2))=(-\infty,m)$. Since for all $t\in(-\infty,m)$, we have
$g(t)<-4<\ell$,  then $g(t)\notin (\ell,-2)$, hence
$g(t)\in(-\infty,\ell)$. We only need, therefore to prove the claim
in $t\in (-\infty,m)$. We proceed by contradiction. Consider
$t_0\in(-\infty,m)$  and suppose that none iterate $t_n\in[m,\ell]$,
so that for all $n>0$ we have $t_n\in(-\infty,m)$. From (iv), the
sequence $\{t_n\}$ is increasing, and as we are assuming that it is
bounded from above by $m$, the sequence must have a limit that, by
continuity,  must be a fixed point, which is a contradiction because
there is no fixed point in  $(-\infty,m]$. Hence the claim is
proved, and we only have to study the dynamics of $g$ in
$[m,\ell]$.

 \noindent (C) We study now the dynamics on the interval $[m,\ell]$.
We denote $m_0:=m$ and $\ell_0:=\ell$, and consider the sequences
\begin{equation}\label{E-g2k}
\begin{array}{l}
 \ell_k=g(m_{k-1})=g^2(\ell_{k-1}) \,\mbox{ and }
   m_k=g(\ell_k)=g^2(m_{k-1}).
\end{array}
\end{equation}
Observe that since $g$ is strictly decreasing in $[m,\ell]$, for $k\geq 1$ we obtain
$$
\left[m_{k-1},{\ell}_k\right]:= g^{2k-1}\left([m ,\ell]\right) \,\mbox{ and }\,
\left[m_{k},{\ell}_k\right]:= g^{2k}\left([m ,\ell]\right).
$$

We will prove that $\{m_k\}$ and $\{\ell_k\}$ are increasing and
decreasing sequences, respectively, that converge to the fixed point
$p$, thus proving the result.

Some computations show that $\ell_1=g(m)=-4<\ell=\ell_0$, and that $m_0=m<m_1=g(\ell_1)=g(-4)$. We proceed by induction,
assuming that $m_{k-1}<m_k$ and $\ell_k<\ell_{k-1}$.
Using  \eqref{E-g2k}, as
 $g$ is decreasing and $m_{k-1}<m_k$, we have $
\ell_k=g(m_{k-1})>g(m_k)=\ell_{k+1}.$
Likewise, since $\ell_{k+1}<\ell_k$ we have
$
m_{k+1}=g(\ell_{k+1})>g(\ell_k)=m_k$. Therefore, the sequences
 $\{m_k\}$ and $\{\ell_k\}$ are monotonous increasing and
  decreasing, respectively. Since both sequences are bounded,
   and  using the expressions in \eqref{E-g2k}, we have that both converge to a fixed point of
 $g^2$. But since there are not  $2$-periodic points, except
  the fixed point  $p$, we have $\lim\limits_{k\to \infty} m_k=\lim\limits_{k\to \infty} \ell_k=p.$~
 \end{proof}

\subsection{Local expression of the unstable manifold}

In order to search numerically the homoclinic points  associated to
$P_2$, we compute an approximation of the local unstable manifold of
the saddle point $P_2=(a_2,b_2)$. We consider the change
 $u=a-a_2$ and $v=b-b_2$, which brings  $P_2$ to the origin $(0,0)$. We also consider the map
 $\widetilde{G}(u,v)=G(u+a_2,v+b_2)-(a_2,b_2)$ which is conjugate with $G$, and
  the linear map given by
 $H(r,s)=L \cdot (r,s)^t$,
 where $L$  is the matrix formed by the eigenvectors of
 $\mathrm{D}G(P_2)$. Hence
 $$L^{-1} \cdot \mathrm{D}G(P_2) \cdot L=
   \left( \begin{array}{cc}
   \lambda_1 & 0 \\ 0 & \lambda_2
   \end{array} \right),$$
 where $\lambda_1$ and $\lambda_2$ are the eigenvalues of  $\mathrm{D}G(P_2)$, given by
        {\footnotesize \begin{align*}
         \lambda_1&=\frac{-1}{384}\left(\left(7\,
         \sqrt {177}-111\right) \,A^{2/3}  +\left(8\,\sqrt {177}-264\right)\,A^{1/3}-768\right)\simeq 7.0701,\\
         \lambda_2&={\frac{-1}{1152}}
         \left(\left(\sqrt {177}
             -25\right)\, A^{2/3}+\left(8\,\sqrt {177}
             -136\right)\,A^{1/3}+1280\right)\simeq -0.4470,
        \end{align*}}
 where $A:=172+12\,\sqrt {177}$.

We compute the Taylor development of the unstable manifold
associated to the origin of the map $ F(r,s)=H^{-1} \circ
\widetilde{G} \circ H(r,s)=\left(\lambda_1
r+O(||(r,s)||^2),\lambda_2 s+O(||(r,s)||^2\right). $ \emph{The
expression of the local unstable manifold $W^u_{\textrm{loc}}(0,0)$
  of $F$ is $s=w(r)=w_2 r^2+w_3 r^3+w_4 r^4+w_5 r^5+O(r^6),$
  where}
  
{\footnotesize \begin{align*}
   &w_2\simeq- 0.00259107002218996975513519324145,\,
   &&w_3\simeq- 0.00013220529650666650558465802906, \\
   &w_4\simeq- 0.00000889870356674847560384348601, \,
   &&w_5\simeq- 0.00000069374812274441343473691330.
  \end{align*}}
These coefficients have been computed using the formulas  in Lemma
\ref{var_inest_expr} of the Appendix, by using floating-point
arithmetic with $60$ digits in the mantissa.
 Observe that we can parametrize  $W^u_{\textrm{loc}}(P_2)$
using the function  $s=w(r)$, by considering
\begin{equation}\label{E-paramWu}
  r\longrightarrow H(r,w(r))+P_2,\, \mbox{ for } r\simeq 0.
\end{equation}
We use this parametrization to obtain Figures 4 and 5.

Finally, from the expression of the local unstable  manifold of the
origin for the map $F$, we obtain that
 the points  $(a,b)\in W^u_{\mathrm{loc}}(P_2)$ satisfy
$
  w(H_1^{-1}(a-a_2,b-b_2))-H_2^{-1}(a-a_2,b-b_2)=0,
$
that can be approximated by
$$
 D_1(a,b):=\sum\limits_{i=2}^5 w_i \big(H_1^{-1}(a-a_2,b-b_2)\big)^i -H_2^{-1}(a-a_2,b-b_2)=0,
$$
where $D_1(a,b)$ is a polynomial of degree 5 that we do not
explicite for the sake of shortness, see \cite[Chap. 5]{Llor} for
more details.

\subsection{Computation of the homoclinic point}

Previous to find a homoclinic point we remember that, by Proposition
\ref{P-atraccio-g}, any point $(a,b)$ such that there exists
$k\in\N$  verifying  $G^k(a,b)\in R_1=\{a-b=0\} \cap
C_1\setminus\{(-1,-1)\}$ belongs to the stable set of $P_2$, since
$G^{k+1}(a,b)\in L_2$.  In this sense, we have graphically observed
that, except for the  point $P_2$, there is no intersection of
$W^u_{\mathrm{loc}}(P_2)$ with the curve $L_2$.  Also we have
observed neither  intersections of $W^u_{\mathrm {loc}}(P_2)$ with
$R_1=\{a-b=0\}$ at the region $C_1$, nor points $(a,b)\in
W^u_{\textrm{loc}}(P_2)$ such that $G(a,b)\in R_1\cup C_1$, but we
have seen the existence of at least one point such that $G^2(a,b)\in
R_1\cup C_1$. See Figure 4.

 Imposing  $G_1(a,b)-G_2(a,b)=0$, we find that the points $(a,b)$ such that  $G(a,b)\in R_1$ satisfy:
  $${D_2}(a,b):=\left( ab+5\,a+5\,b+9 \right) ^{3}-
  \left( a+b+6 \right) ^{3} \left(a+b+2 \right) ^{2}=0.$$
 Hence, the points such that $G^2(a,b)\in R_1$  are those satisfying ${D_2}
 (G(a,b))=0$, or equivalently  $D_3(a,b):=\mathrm{numer}({D_2}
 (G(a,b)))=0,$
 where $D_3(a,b)$  is a polynomial of degree $10$ in the variable
 $m=(a+b+2)^{2/3}$ with $22$ terms, that we omit here.

Therefore, the homoclinic point  $P$ must verify the system
 $\{D_1(a,b)=0, D_3(a,b)=0\}$.
We solve it numerically, using floating-point arithmetic with $60$
digits in the mantissa, and we get a solution in
$[-6,-5]\times[3.5,5]$, given by $P=(p_1,p_2)$ where 

{\footnotesize
\begin{align*}
  p_1&\simeq-5.67750144031789435343891174392876990152177028290023619512062, \\
  p_2&\simeq 4.10574868714920935493626045239900450809925741194290963919902 .
 \end{align*}}
By using the parametrization of $W^u_{\mathrm{loc}}(P_2)$ given by
\eqref{E-paramWu}, we find that the point
 $P$ corresponds with the parameter
 $r\simeq -1.48202 15208 77494 33523.$

By construction, $G^3(P)$ which must lie on $L_2$. A computation
shows that the absolute error when we evaluate $R(a,b)$ on this
point, is $\big|R\big(G^3(p)\big)\big|\simeq 10^{-58}$. Accordingly,
  the point $P$ exhibits, numerically, a homoclinic behavior.

As can be seen in Figure 4, there exists another solution  of
$\{D_1(a,b)=0, D_3(a,b)=0\}$ in  $[-8,-6]\times[3.5,5]$, given by
$\tilde{P}=(\tilde{p}_1,\tilde{p}_2)$ where 

{\footnotesize
\begin{align*}
 \tilde{p}_1&\simeq-7.32664831286596004531700787733138125161658087249633041273728, \\
  \tilde{p}_2&\simeq 4.26205920129322448141657538934356322617112224124511704493689.
 \end{align*}
} The point $\tilde{P}$ corresponds to the parameter value $r\simeq
-3.14702 44917 79071 04545$.

 \begin{center}
 \begin{minipage}{0.5\textwidth}
  \includegraphics[scale=0.37]{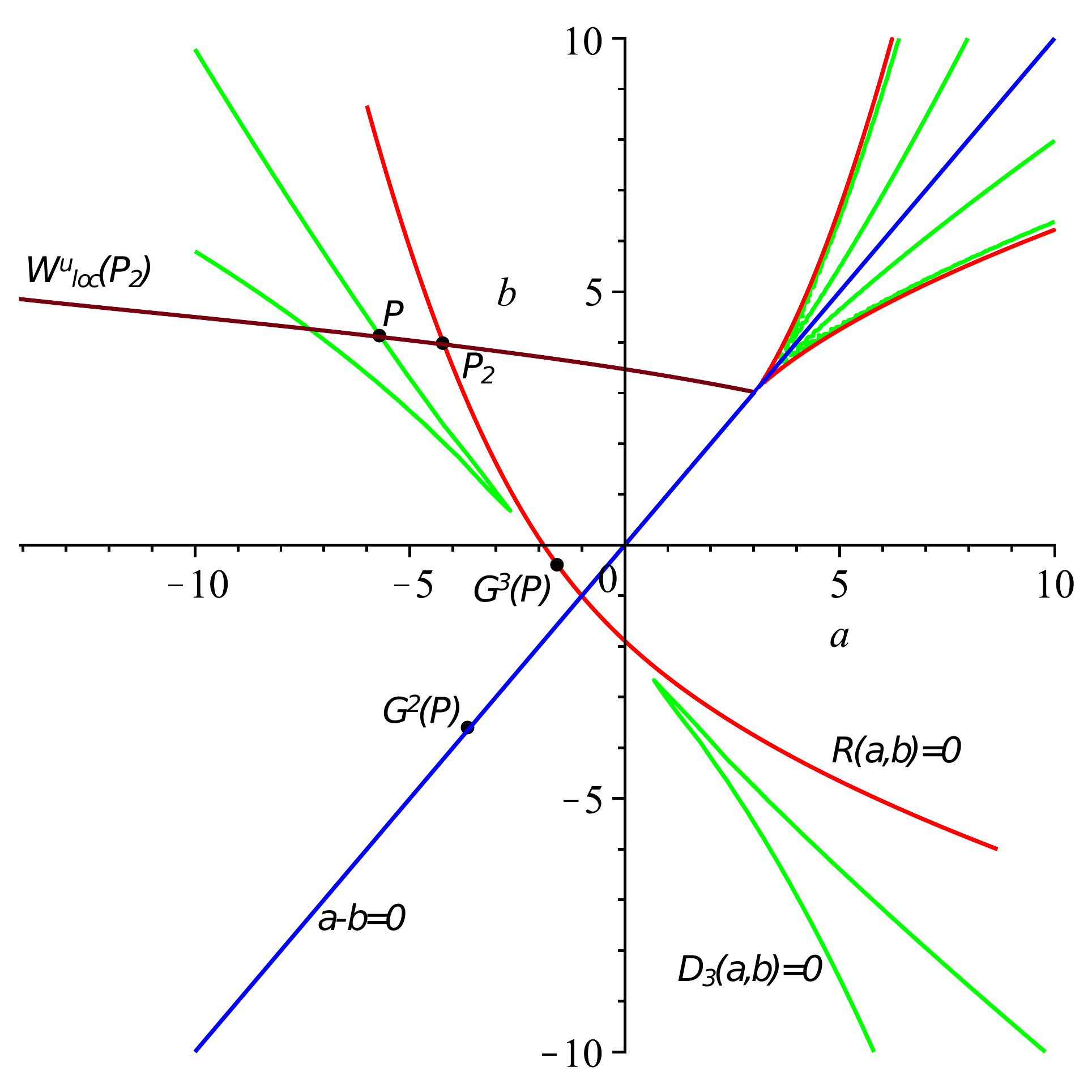}
  \end{minipage}
\begin{minipage}{0.47\textwidth}
  Figure 4. Location of the points $P, G^2(P)$ and
  $G^3(P)$ on $W^u_{\textrm{loc}}(P_2)$ (brown), the curve  $D_3(a,b)=0$ (green),  the diagonal $R_1$ (blue) and
  the resolvent curve (red), respectively. The point $G(P)$ is outside the image.
\end{minipage}
 \end{center}

\subsection{Computation of points in $W^u_{\mathrm{loc}}(P_2)\cap\mathcal{F}$}

To find a point in $W^u_{\mathrm{loc}}(P_2)\cap\mathcal{F}$, we solve numerically the system $
 \left\{D_1(a,b)=0,a+b+2=0\right\},$
obtaining the point $Q=(q_1,q_2)$, where 

{\footnotesize
\[
  q_1\simeq-6.15163017029193114270539883292276699558057876233980350720282,
  \quad q_2=-2-q_1,
\]}
This point corresponds with the parameter $r\simeq -1.96025 81538
61616 87597 .$

To find another point with a parameter value closer to zero (hence giving a better evidence of really being in
$W^u_{\mathrm{loc}}(P_2)$), we find a point $Q_{-1}$ such that $G(Q_{-1})=Q$. The points $(a,b)$ such that
$G(a,b)\in\{a+b+2=0\}$,
 verify
 $$D_4(a,b):=5\,a+5\,b+ab+9+ \left( a+b+6 \right) \sqrt [3]{ \left( a+b+2 \right) ^
  {2}}+2\,\sqrt [3]{ \left( a+b+2 \right) ^{4}}=0.$$
 By solving numerically the system
 $\left\{D_1(a,b)=0,D_4(a,b)=0\right\},$
we find $Q_{-1}:=(z_1,z_2)$ where 

{\footnotesize
\begin{align*}
  z_1&\simeq-4.43931733951927306713914976146761550810750048579478327758904, \\
  z_2&\simeq 3.98185284365899589972467095578564600569428848801825836848384.
 \end{align*}}
The point $Q_{-1}$ has an associated parameter value
 $r\simeq -0.23505 95678 85428 61108.$ The location of the above points is shown in Figure 5.

Observe that the parameters of the points $Q$, $P$ and $Q_{-1}$ are
interspersed, so the points are also interspersed in
$W^u_{\mathrm{loc}}(P_2)$. An analytic proof of this fact would show
that arbitrarily near of
 $P_2$ there are homoclinic points and points in
$\mathcal{F}$.

 \begin{center}
 \begin{minipage}{0.5\textwidth}
  \includegraphics[scale=0.37]{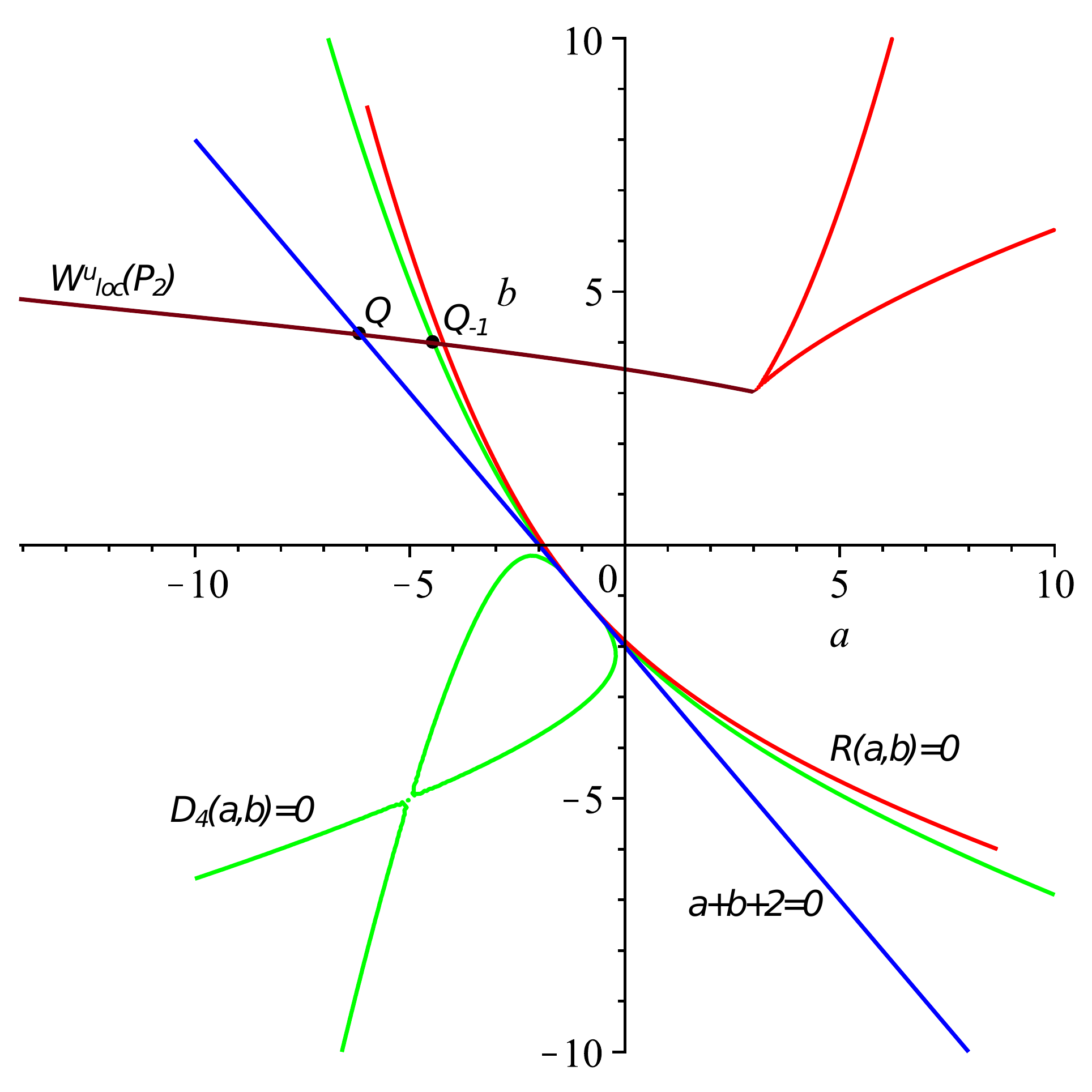}
  \end{minipage}
\begin{minipage}{0.47\textwidth}
 Figure 5. Location of the points $Q$ and $Q_{-1}$
  in $W^u_{\mathrm{loc}}(P_2)\cap\mathcal{F}$; the curve $D_4(a,b)=0$ (green);
  the line $a+b+2=0$ (blue); and the resolvent curve
  (red), respectively.
\end{minipage}
 \end{center}

\section*{Appendix: Local unstable manifold near a hyperbolic saddle point}

 \begin{lem}\label{var_inest_expr}
Consider the smooth map, defined in a neighborhood of the origin \
$\mathcal{U}$:
$$
F(x,y)=\left(\lambda x+ \sum_{{i+j}={2}}^5{f_{i,j}
  x^i y^j}+O(||(x,y)||^6),\mu y+ \sum_{{i+j}={2}}^5{g_{i,j} x^i y^j}+O(||(x,y)||^6)\right)
$$
  where $\arrowvert \lambda \arrowvert > 1 >
  \arrowvert\mu\arrowvert$, so that the origin is a hyperbolic saddle. Let
  $y=w(x)=\sum_{k=2}^5{w_k x^k}+O(x^6)$ be the expression of the local
  unstable manifold  in a neighborhood of the origin. Then:
  $$ w_2=\frac{g_{2,0}}{\lambda^2-\mu},\quad w_3={\frac
{{\lambda}^{2}g_{{3,0}}-2\,\lambda\,f_{{2,0}}g_{{2,0}}-\mu\,g_{
{3,0}}+g_{{1,1}}g_{{2,0}}}{ \left( {\lambda}^{2}-\mu \right)  \left(
{ \lambda}^{3}-\mu \right) }} ,\quad
  w_4=\frac{W_4}{ \left( {\lambda}^{2}-\mu \right) ^{2} \left( {\lambda}^{3}-\mu
 \right)  \left( {\lambda}^{4}-\mu \right)} ,
  $$
 where
\noindent\begin{align*} W_4&=  g_{{4,0}}{\lambda}^{7}+ \left(
-3\,f_{{2,0}}g_{{3,0}}-2\,f_{{3,0}}g_{{ 2,0}} \right) {\lambda}^{6}+
\left( 5\,f_{2,0}^2g_{{2,0}}-2\,g_{
{4,0}}\mu+g_{{2,0}}g_{{2,1}} \right) {\lambda}^{5}\\
&+ \left(  \left( 6\,
f_{{2,0}}g_{{3,0}}+2\,f_{{3,0}}g_{{2,0}}-g_{{4,0}} \right) \mu-2\,f_{{
1,1}}g_{2,0}^2-3\,f_{{2,0}}g_{{1,1}}g_{{2,0}}+g_{{1,1}}g_{{3,0}}
 \right) {\lambda}^{4}\\
 &+ \left( g_{{4,0}}{\mu}^{2}+ \left( -5\,
 f_{2,0}^{2}g_{{2,0}}+2\,f_{{3,0}}g_{{2,0}}-g_{{2,0}}g_{{2,1}} \right) \mu-
2\,f_{{2,0}}g_{{1,1}}g_{{2,0}}+g_{{0,2}}g_{2,0}^2 \right) {
\lambda}^{3}\\
&+ \left(  \left( -3\,f_{{2,0}}g_{{3,0}}+2\,g_{{4,0}}
 \right) {\mu}^{2}+ \left( f_{2,0}^2g_{{2,0}}+3\,f_{{2,0}}g_{{1,
1}}g_{{2,0}}-2\,g_{{1,1}}g_{{3,0}}-g_{{2,0}}g_{{2,1}} \right) \mu\right.\\
&\left.+g_{1,1}^{2}g_{{2,0}} \right) {\lambda}^{2}+ \left( -2\,f_{{3,0}}g_{{2,0
}}{\mu}^{2}+ \left( 2\,f_{{1,1}}g_{2,0}^2+2\,f_{{2,0}}g_{{1,1}}g
_{{2,0}} \right) \mu \right) \lambda\\
&-g_{{4,0}}{\mu}^{3}+ \left( -f_{{2,0}}^{2}g_{{2,0}}+g_{{1,1}}g_{{3,0}}+g_{{2,0}}g_{{2,1}} \right) {\mu
}^{2}+ \left( -g_{{0,2}}g_{2,0}^2-g_{1,1}^2g_{{2,0}}
 \right) \mu
\end{align*}
and
$$
  w_5=\frac{\sum_{i=0}^{13} p_i\,\lambda^i}{\left( {\lambda}^{2}-\mu \right) ^{3} \left( {\lambda}^{3}-\mu
 \right)  \left( {\lambda}^{4}-\mu \right)  \left( {\lambda}^{5}-\mu
 \right)},
  $$
 where $p_i, i=1,2,\ldots,13$ are polynomials in the other variables
 of $F$ that we skip, although we have used, for the sake of
 shortness (they are given in \cite[Chapter 5]{Llor}).
\end{lem}

 \begin{proof} Due to the particular form of the linear part of  $F$, the local unstable manifold $W^u_{\mathrm{loc}}(0,0)$ is given by
a smooth function of the form
$y=w(x)=w_2x^2+w_3x^3+w_4x^4+w_5x^5+O(x^6),$ that is, a  point is on
the local stable manifold if it is of the form
 $(x,w(x))$. Imposing that $F(x,w(x))=(F_1(x,w(x)),F_2(x,w(x)))$ is also on this curve we get  that
  the points on the local unstable manifold must satisfy
$
  F_2(x,w(x))=w\big(F_1(x,w(x))\big).
$
The result follows by comparing the terms in the Taylor development
of both members of the last equation.~\end{proof}

\subsection*{Acknowledgements}  The authors are supported by
Ministry of Economy, Industry and Competitiveness of the Spanish
Government through grants MINECO/FEDER MTM2016-77278-P  (first
author) and DPI2016-77407-P
 (AEI/FEDER, UE, second and third author). The first  author is also supported by the grant 2014-SGR-568  from
AGAUR,  Generalitat de Catalunya. The third author is supported by
the grant 2014-SGR-859 from AGAUR, Generalitat de Catalunya.

\end{document}